\documentclass[a4paper,oneside,10pt]{article}

\usepackage{mathtools,amsmath,amssymb,amsthm,amsopn,amsfonts,newpxtext,newpxmath,tensor}
\usepackage{tikz-cd, tikz, quiver}
\usepackage{enumerate}
\usepackage[colorlinks=true, allcolors=blue]{hyperref}
\usepackage{url}
\usepackage{geometry}
\usepackage{setspace}
\allowdisplaybreaks

\theoremstyle{definition}
\newtheorem{prop}{Proposition}[section]
\newtheorem{lem}[prop]{Lemma}
\newtheorem{thm}[prop]{Theorem}

\newtheorem{constr}[prop]{Construction}
\newtheorem{cor}[prop]{Corollary}
\newtheorem{defn}[prop]{Definition}
\newtheorem{rem}[prop]{Remark}
\newtheorem{expl}[prop]{Example}
\newtheorem{notation}[prop]{Notation}

\newcommand{\op}{\text{op}}
\newcommand{\ev}{\mathrm{ev}}
\newcommand{\coev}{\mathrm{coev}}

\newcommand\isoto{\stackrel{\sim}{\smash{\longrightarrow}\rule{0pt}{0.4ex}}}

\DeclareFontFamily{U}{dmjhira}{}
\DeclareFontShape{U}{dmjhira}{m}{n}{ <-> dmjhira }{}
\DeclareRobustCommand{\yo}{\text{\usefont{U}{dmjhira}{m}{n}\symbol{"48}}}

\title{Quartic BV structures in supercategories and modified necklace Lie bialgebras}
\author{Nikolai Perry{\footnote{Email: \texttt{Nikolai.Perry@unige.ch}}}\\\textit{Section de Mathématiques, Université de Genève, Switzerland}}
\date{\today}

\begin{document}

\maketitle

\begin{abstract}
    We introduce a modified version of the necklace Lie bialgebra associated to a quiver, in which the bracket and cobracket insert (rather than remove) pairs of arrows in involution. This structure is then related to canonical quartic Poisson/Batalin-Vilkovisky structures on suitable representation varieties of the quiver. Constructions on the representation side take place in symmetric monoidal $\Pi$-categories, which prompts a discussion of graded differential operators on commutative monoids in any such category. The generality of the categorical approach allows us to fully recover necklace structures, showing how the modified and classical necklace operations are related via dualisability.
\end{abstract}

\tableofcontents

\section{Introduction}
The free vector space $\mathcal{N}$ on the set of cyclic paths in a double quiver carries a canonical involutive Lie bialgebra (IBL) structure in which the bracket $\mathrm{br}$ and cobracket $\delta$ remove pairs of arrows in involution \cite{ginzburg2001,bocklandt2002,schedler2005hopf}. Pictorially,
\begin{equation*}
    \vcenter{\hbox{\includegraphics[scale=1]{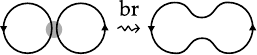}}} \quad \text{and} \quad \vcenter{\hbox{\includegraphics[scale=1]{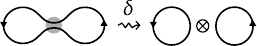}}}\;.
\end{equation*}
This is the \textit{necklace Lie bialgebra}, an infinitesimal analogue of the Goldman-Turaev Lie bialgebra of a surface \cite{goldman1986invariant,turaev1991skein,chas2004combinatorial}, where cyclic paths play the role of loop homotopy classes. The quiver and surface stories are related by so-called \textit{expansions} --- see \cite[\S1.2]{kawazumi2014logarithms} or \cite[Thm.\ 6.17]{alekseev2018goldman} and references therein.

One reason for interest in the Goldman-Turaev Lie bialgebra is its relation to character varieties. Its underlying Lie algebra, itself regarded as the origin of string topology \cite{chas1999string}, is related to the Atiyah-Bott Poisson structure on the $GL_N$-character variety of the surface via the Goldman map \cite{goldman1986invariant}, while its full IBL structure is related to a Batalin-Vilkovisky (BV) structure on the $Q_N$-character supervariety (where $Q_N$ is the queer Lie supergroup) through an odd analogue of this map \cite{alekseev2024batalin}. Likewise, the necklace Lie bialgebra is related to representation varieties of the double quiver (which can be viewed as infinitesimal analogues of character varieties): the Lie bracket $\mathrm{br}$ is witnessed through (constant-coefficient) Poisson structures on ordinary representation varieties \cite{ginzburg2001,bocklandt2002}, while the full IBL structure $(\mathrm{br},\delta)$ --- which can be encoded as a BV operator on $\bigwedge \mathcal{N}$ --- is witnessed through (constant-coefficient) BV structures on appropriate representation (super)varieties \cite{barannikov2014matrix,perry2024graded}.

In this paper we introduce a seemingly-overlooked variant of the necklace Lie bialgebra. Surprisingly, by tweaking the necklace operations only slightly, to operations $(\mathrm{br}^+,\delta^+)$ that \textit{insert} pairs of arrows in involution, the delicate IBL (and hence BV) structure is maintained. We call this the \textit{augmented necklace Lie bialgebra} and, for homogeneity, denote the classical operations $(\mathrm{br},\delta)$ by $(\mathrm{br}^-,\delta^-)$. Pictorially, the augmented operations are
\begin{equation*}
    \vcenter{\hbox{\includegraphics[scale=1]{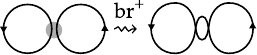}}} \quad \text{and} \quad \vcenter{\hbox{\includegraphics[scale=1]{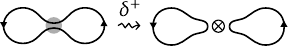}}}\;.
\end{equation*}
To the best of our knowledge, this is the only modification of this type that maintains the IBL structure, and the rest of this paper was initially motivated by a desire to understand where this structure comes from. On the one hand, the bracket $\mathrm{br}^+$ can be lifted to a Poisson double bracket on the path algebra of the double quiver, one that is \textit{quartic} in terms of generators of the path algebra \eqref{eqn:double-bracket}. While we believe this double bracket to be interesting in its own right, `double' theoretic aspects are not the focus of this paper. On the other hand, the Lie algebra $(\mathcal{N},\mathrm{br}^+)$ and IBL algebra $(\mathcal{N},\mathrm{br}^+,\delta^+)$ are witnessed as \textit{quartic}-coefficient Poisson and BV structures on representation varieties, respectively, and this is what we investigate.

We work in a more-general context on the representation side, in which `representation varieties' and their `algebras of functions' exist internal to any suitable linear symmetric monoidal category, as opposed to just the category of (super) vector spaces. Let us sketch the utility of this. Given a `representation variety' $h$ in such a category $\mathcal{C}$, the `true' representation variety is really $\mathcal{C}(1,h)$, its space of $1$-points. The `algebra of functions' on $h$ is naively modelled by $Sh^\vee$, the free commutative monoid on its dual. However, its space of $1$-points $\mathcal{C}(1,Sh^\vee)$ looks nothing like the algebra of functions on $\mathcal{C}(1,h)$ in general: \textit{it can be far richer}. This is, in essence, what lets us fully recover necklace structures. 

For us, the correct categorical context for quiver representation varieties is that of \textit{symmetric monoidal $\Pi$-categories} in the sense of \cite{brundan2017monoidal}, linear symmetric monoidal categories equipped with `parity' structure that allows one to speak of even and odd maps in a workable way --- the tautological example being the category of super vector spaces. This necessitates a notion of (graded) differential operators on commutative monoids in any symmetric monoidal $\Pi$-category, a generalisation of those on commutative (super)algebras in the sense of Koszul \cite{koszul1985crochet}. While this theory works essentially as expected, a large part of our preliminaries section is devoted to it.

Introduced for single-vertex quivers in \cite{barannikov2014matrix}, and then later for arbitrary quivers in \cite{perry2024graded} (the `$p=1$' case in that paper), there is a natural graded variant of the necklace Lie bialgebra, denoted $(\mathcal{N}_\text{gr},\mathrm{br}_\text{gr}^-,\delta_\text{gr}^-)$ in this text, that is better adapted to the BV story. By considering BV structures on generalised representation varieties into which we show that $S\mathcal{N}_\text{gr}$ can inject, we derive the augmented variant of this structure, $(\mathcal{N}_\text{gr},\mathrm{br}^+_\text{gr},\delta^+_\text{gr})$, and prove that it satisfies (the appropriate version of) IBL algebra axioms.

A natural open question is whether the augmented necklace story has an analogue in the surface setting. If so, one might expect to find interesting new structures on character (super)varieties. Now, for $G$ a Lie group with metric Lie algebra $\mathfrak{g}$, the Atiyah-Bott Poisson structure on the $G$-character variety $\mathrm{Ch}_G (\Sigma)$ of a surface $\Sigma$ with nonempty boundary can be understood \textit{locally}, in the sense of factorisation homology, as coming from the infinitesimal braiding on the category $\mathrm{Rep}G$ of locally finite-dimensional $\mathrm{U}\mathfrak{g}$-modules. Namely, it was shown in \cite{ben2018integrating} that the algebra of functions $\mathscr{O}(\mathrm{Ch}_G(\Sigma))$ arises as the $G$-invariant part of the internal endomorphism algebra for the module category
\begin{equation*}
    \int_\Sigma \mathrm{Rep}G,
\end{equation*}
where the link between deformation $\mathrm{Rep}G \rightsquigarrow \mathrm{Rep}_qG$ and quantisation of this algebra was also studied. Expanded upon in \cite{karlsson2024deformation}, first-order deformation of $\mathrm{Rep}G$ given by the infinitesimal braiding can be tracked through factorisation homology in a systematic way, where it is seen to deform the commutative product on $\mathscr{O}(\mathrm{Ch}_G(\Sigma))$ in the direction of the Atiyah-Bott Poisson structure. Thus, we can push our speculation further and ask whether these new structures on character (super)varieties --- if they exist --- can be understood locally in a similar way.

\subsection*{Conventions}
In all of what follows, we fix a commutative ground-ring $R$. If unspecified, all linear categories, linear functors, modules, algebras, etc.\ are understood to be over $R$. We assume that 2 $\in R^\times$ is invertible, although this is only needed onwards from \S\ref{subsec:lie-bialg-structure}. Categories are assumed to be locally small. A symmetric monoidal category is called \textit{$\otimes$-cocomplete} if it is possesses all (small) colimits and if its monoidal product $\otimes$ perserves them (in each variable). We use string diagram notation for maps in symmetric monoidal categories (which we always treat as strict for simplicity), reading composition bottom-to-top.

\subsection*{Outline and Summary of results}
In \S\ref{sec:The-setting} we recall background material and define the superalgebra $\underline{\mathrm{D}}(a)$ of \textit{(graded) differential operators} on a commutative monoid $a$ in a symmetric monoidal $\Pi$-category $\mathcal{C}$, which is filtered by \textit{order}, proving its basic properties. The upshot for us is that we can speak of (and work with) \textit{BV algebras} in any such category. When $\mathcal{C}$ is $\otimes$-cocomplete and $a=Sx$ is free on an object $x \in \mathcal{C}$, we construct an explicit isomorphism
\begin{equation*}
    \underline{C}(S^{\leq k}x,Sx) \cong \underline{\mathrm{D}}_{k}(Sx)
\end{equation*}
between the supermodule of even \textit{and odd} maps $S^{\leq k}x \to Sx$ and the supermodule of (graded) differential operators on $Sx$ of order $k$ (Theorem \ref{thm:!existence-of-diffops-graded}).
\medskip

Fixing a double quiver $\overline{Q}$, we introduce its augmented necklace Lie algebra $(\mathcal{N},\mathrm{br}^+)$ in \S\ref{sec:NLBs}. Given a $\otimes$-cocomplete linear symmetric monoidal category $\mathcal{C}$ and dualisable objects $c_i \in \mathcal{C}$ for each vertex $i$ of the quiver, we construct an object $h \in \mathcal{C}$ (a `representation variety') and show that the commutative monoid $Sh^\vee$ (its `algebra of functions') admits two canonical Poisson brackets $\{-,-\}^-$ and $\{-,-\}^+$, pictorially defined by
\begin{equation*}
    \vcenter{\hbox{\includegraphics[scale=1]{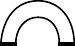}}} \quad \text{and} \quad \vcenter{\hbox{\includegraphics[scale=1]{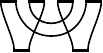}}}\;,
\end{equation*}
respectively. In particular, the commutative $R$-algebra $\mathcal{C}(1,Sh^\vee)$ of maps from the monoidal unit $1$ to $Sh^\vee$ admits two canonical Poisson brackets too.\footnote{This is because linear lax symmetric monoidal functors, such as $\mathcal{C}(1,-)$, send Poisson algebras to Poisson algebras.} Writing $S\mathcal{N}$ for the free $R$-algebra on $\mathcal{N}$ (which has Poisson brackets induced by the Lie brackets $\mathrm{br}^\pm$), we define an algebra map
\begin{equation*}
    \mathrm{tr}:S\mathcal{N} \to \mathcal{C}(1,Sh^\vee)
\end{equation*}
that intertwines $\mathrm{br}^{\pm}$ with $\{-,-\}^\pm$. We show that this map is injective for $\mathcal{C}$ a category of certain 1-dimensional cobordisms (Proposition \ref{prop:injectivity-of-tr}), showing that the simple Poisson brackets on $Sh^\vee$ subsume the combinatorics of the (augmented) necklace Lie algebra.
\medskip

In \S\ref{subsec:lie-bialg-structure} we consider the full IBL structure $(\mathcal{N},\mathrm{br}^-,\delta^-)$ of the necklace Lie bialgebra, introducing its augmented variant $(\mathcal{N},\mathrm{br}^+,\delta^+)$. These IBL structures can be understood as BV operators $\Delta^\pm$ on $\bigwedge \mathcal{N}$. Given a $\otimes$-cocomplete symmetric monoidal $\Pi$-category $\mathcal{C}$ and dualisable objects $c_i \in \mathcal{C}$ with odd involutions $\iota_i \in \underline{\mathcal{C}}(c_i,c_i)$ for each vertex, we consider a commutative monoid $S\widetilde{h}^\vee$ (where $\widetilde{h}$ is an appropriate `representation variety') and show that it --- and hence the commutative $R$-superalgebra $\underline{\mathcal{C}}(1,S\widetilde{h}^\vee)$ of even and odd maps from $1$ to $S\widetilde{h}^\vee$\footnote{Using that lax symmetric monoidal $\Pi$-functors send BV algebras to BV algebras (Proposition \ref{prop:pi-functor-on-BV-algs}).} --- carries two canonical BV operators $\widetilde{\Delta}^\pm$. We define a superalgebra map
\begin{equation*}
    \textstyle \mathrm{otr}:\bigwedge \mathcal{N} \to \underline{\mathcal{C}}(1,S\widetilde{h}^\vee),
\end{equation*}
and show that it intertwines the BV operators $\Delta^\pm$ and $\widetilde{\Delta}^\pm$. Due to the presence of the involutions $\iota_i$ in the constructions, the question of injectivity of $\mathrm{otr}$ is harder, and is left to \S\ref{sec:Injectivity-of-otr}.
\medskip

More natural for the BV approach are graded variants $(\mathcal{N}_\text{gr},\mathrm{br}_\text{gr}^\pm,\delta_\text{gr}^\pm)$ of the (augmented) necklace Lie bialgebra mentioned earlier, and these are the concern of \S\ref{sec:grNLBs}. These yield BV operators $\Delta^\pm_\text{gr}$ on $S\mathcal{N}_\text{gr}$. Given a $\otimes$-cocomplete symmetric monoidal $\Pi$-category $\mathcal{C}$ and dualisable objects $c_i \in \mathcal{C}$ (here the odd involutions are \textit{not} needed) we as before show that the commutative monoid $Sh_\text{gr}^\vee$ (where $h_\text{gr}$ is again an appropriate `representation variety') admits two canonical BV operators $\nabla^\pm$, and we give a
natural superalgebra map
\begin{equation*}
    \mathrm{tr}_\text{gr}:S\mathcal{N}_\text{gr} \to \underline{\mathcal{C}}(1,Sh_\text{gr}^\vee)
\end{equation*}
that intertwines the BV operators $\Delta^\pm_\text{gr}$ and $\nabla^\pm$. In fact, we work backwards, \textit{deriving} the operations $\mathrm{br}_\text{gr}^\pm$ and $\delta_\text{gr}^\pm$ --- and simultaneously proving that they satisfy the appropriate version of IBL algebra axioms (Corollary \ref{cor:-grneck-is-odd-ibl}) --- from the representation perspective, which is made possible by the following result that we also prove (in the proof of Theorem \ref{thm:graded-necklace-BV-alg}): \textit{the map $\mathrm{tr}_\text{gr}$ is injective for a certain choice of category and dualisable objects }(again, this concerns 1-dimensional cobordisms). This, in our view, justifies the categorical approach, perhaps hinting at it as more fundamental.
\medskip

Finally, in \S\ref{sec:Injectivity-of-otr}, we address the question of whether the map $\mathrm{otr}$ can be injective. We suspect the answer is yes, but we do not prove it. Instead, we give an example where $\mathrm{otr}$ has kernel the ideal generated by constant paths (Proposition \ref{prop:near-injectivity-of-otr}), and briefly outline how one might arrange for $\mathrm{otr}$ to be injective.

\subsection*{Acknowledgements}
It is my pleasure to thank the following people: Ján Pulmann, for insightful discussions, his interest in this work, and valuable advice on the writing of this paper; Lucy Spouncer, for the idea of considering arrow-insertion as opposed to arrow-removal --- which arose during discussions with homotopy theorists in Edinburgh --- and for help in verifying IBL algebra conditions; Anton Alekseev and Pavol Ševera, for enlightening conversations and ideas.

This research was supported in part by the grant number 226683 and by the National Center for Competence in Research SwissMAP of the Swiss National Science Foundation.

\section{Preliminaries}\label{sec:The-setting}

The purpose of this section is mainly to find the right categorial setting needed to discuss \textit{BV algebras}, setting up machinery to define them. We discuss more than what will be needed though, in order to show where they sit in the larger mathematical context.

Given an object $x$ in a symmetric monoidal category $\mathcal{C}$, the \textit{free commutative monoid on $x$ in $\mathcal{C}$} (if it exists) is a commutative monoid $Sx$ in $\mathcal{C}$ together with a map $x \to Sx$, initial as such. When $Sx$ exists for all $x$, then this construction can of course be extended to a functor $S:\mathcal{C} \to \mathrm{CMon}(\mathcal{C})$ left adjoint to forget. Our constructions in later sections involve looking at \textit{differential operators} on free commutative monoids, so we recall the following.
\begin{constr}\label{constr:free-com-mon}
    Suppose that the symmetric monoidal category $\mathcal{C}$ is $\otimes$-cocomplete. Then the free commutative monoid on any object $x$, $Sx$, exists and can be constructed explicitly in the following standard way:
    \begin{itemize}
        \item \textit{The underlying object.}\\
        For each $n \geq 0$, the symmetric group $S_n$ acts on the object $x^{\otimes n}$ by permuting monoidal factors. Define $S^n x = (x^{\otimes n})_{S_n}$ to be the coequalizer of this action. That is, it is the colimit
        \begin{equation}
            \begin{tikzcd}[cramped]
            	{x^{\otimes n}} & {x^{\otimes n}} & {S^n x}
            	\arrow[""{name=0, anchor=center, inner sep=0}, shift right=2, from=1-1, to=1-2]
            	\arrow[""{name=1, anchor=center, inner sep=0}, "{(S_n)}", shift left=2, from=1-1, to=1-2]
            	\arrow[from=1-2, to=1-3]
            	\arrow[shorten <=1pt, shorten >=1pt, dotted, no head, from=1, to=0].
            \end{tikzcd}
        \end{equation}
        In particular, we have $S^0 x \cong 1$ and $S^1 x \cong x$. The underlying object of the free commutative monoid on $x$ is given by the infinite coproduct
        \begin{equation}\label{eqn:graded-decomposition}
            Sx = \bigsqcup_{n\geq 0} S^n x.
        \end{equation}

        \item  \textit{The unit.}\\
        The unit map $1 \to Sx$ is simply the stuctural map $1 \cong S^0 x \to Sx$.

        \item \textit{The multiplication.}\\
        Since $\otimes$ preserves colimits, we can specify the multiplication map $Sx \otimes Sx \to Sx$ by specifying maps $S^m x \otimes S^n x \to S^{m+n} x$ for each $m,n \geq 0$. To specify these maps, notice that $x^{\otimes m} \otimes x^{\otimes n} \cong x^{\otimes (m+n)} \to S^{m+n}x$ coequalises the action of $S_m \times S_n$, of which $S^m x \otimes S^n x$ is the coequaliser (since $\otimes$ preserves colimits). Therefore, the maps $x^{\otimes m} \otimes x^{\otimes n} \to S^{m+n}x$ factor, yielding the desired maps $S^m x \otimes S^n x \to S^{m+n} x$.
    \end{itemize}
    This data renders $Sx$ a commutative monoid in $\mathcal{C}$, and it is quick to check that the structural map $x \cong S^1 x \to Sx$ is initial among all maps from $x$ to a commutative monoid in $\mathcal{C}$. Explicitly, if $f:x \to a$ is any other such map, then the unique monoid map $Sx \to a$ factoring it has $S_n$-coequalising lifts
    \begin{equation*}
        \begin{tikzcd}
        	{x^{\otimes n}} & {a^{\otimes n}} & a
        	\arrow["{f^{\otimes n}}", from=1-1, to=1-2]
        	\arrow["\mu", from=1-2, to=1-3]
        \end{tikzcd},
    \end{equation*}
    where $\mu$ is the multiplication of $a$. Notice that $Sx$ is in fact graded according to the decomposition \eqref{eqn:graded-decomposition}, which we refer to as `polynomial grading'. 
\end{constr}

\subsection{Differential operators on commutative monoids in linear SMCs}\label{sec:Diffops}
The notion of differential operators on commutative algebras makes sense for commutative monoids in any linear symmetric monoidal category. This has appeared in the literature before, such as in \cite[\S2.2]{ginzburg2010differential}, but we give our description, together with explicit derivations of basic properties for completeness (which are analogous to the usual theory).\footnote{While we restrict ourselves to commutative settings, noncommutative settings are considered in \cite{ginzburg2010differential}. In the ordinary case of the category of (super) vector spaces, nonassociative settings are considered in \cite{akman1997some}.}

Let $\mathcal{C}$ be an $R$-linear symmetric monoidal category and $a \in \mathcal{C}$ a commutative monoid. For $m,n \geq 1$, consider the $R$-linear maps $\mathcal{C}(a^{\otimes m},a^{\otimes n}) \to \mathcal{C}(a^{\otimes (m+1)},a^{\otimes n})$ given by
\begin{equation*}
    \vcenter{\hbox{\includegraphics[scale=1]{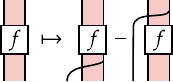}}}\;.
\end{equation*}
We denote these maps by $T$. Consider the sequence of maps
\begin{equation}\label{eqn:sequence-of-maps}
    \mathcal{C}(a,a) \stackrel{T}{\to} \mathcal{C}(a \otimes a ,a) \stackrel{T}{\to} \mathcal{C}(a \otimes a \otimes a,a) \to \cdots
\end{equation}
\begin{defn}\label{defn:diffop}
    An endomorphism $\Delta \in \mathcal{C}(a,a)$ is called a \textbf{differential operator on $a$} if there exists some $k\geq1$ such that $T^k(\Delta) = 0$. In this case, the differential operator $\Delta$ is said to have \textbf{order} ($\leq$) $k-1$.
\end{defn}
An explicit form of the defining condition for a differential operator will be given in Proposition \ref{prop:diffop-condition}.
\begin{rem}
    The composites $T^k(\Delta) \circ (\mathrm{id} \otimes u)$, where $u:1\to a$ is the unit, are generalisations of the maps $\Phi^k_\Delta$ defined in \cite{koszul1985crochet}.
\end{rem}
Let $\mathrm{D}(a) \subset \mathcal{C}(a,a)$ denote the submodule of differential operators on $a$, and let $\mathrm{D}_k(a) \subset \mathrm{D}(a)$ denote the submodule of differential operators of order $k$. We have a filtration
\begin{equation*}
    \mathrm{D}_0(a) \subset \mathrm{D}_1(a) \subset \mathrm{D}_2(a) \subset \dots \subset \mathrm{D}(a).
\end{equation*}
In fact, $\mathrm{D}(a)$ is a filtered \textit{subalgebra} of $\mathcal{C}(a,a)$. Before this is proven, we state an intermediate result.
\begin{lem}\label{lem:T-on-comp}
Let $f: a^{\otimes m} \to a^{\otimes n}$ and $g: a^{\otimes n} \to a^{\otimes p}$ be maps (with $m,n,p \geq 1$). Then for all $N \geq 0$,
    \begin{equation*}
        T^N(g\circ f)
        = \sum_{I \subset \{1,\dots,N\}} T^{|I|}(g) \circ \left(\mathrm{id}_{a^{\otimes |I|}} \otimes T^{N-|I|}(f)\right) \circ \left( \sigma_I \otimes \mathrm{id}_{a^{\otimes m}} \right),
    \end{equation*}
    where $\sigma_I:a^{\otimes N} \to a^{\otimes N}$ permutes the tensor factors according to the unique permutation of $\{1,\dots,N\}$ that maps $I$ to  $\{1,\dots,|I|\}$ and preserves the orderings on both $I$ and $\{1,\dots,N\}\setminus I$ (a shuffle). Diagrammatically,
    \begin{equation}\label{eqn:T-on-comp}
        \vcenter{\hbox{\includegraphics[scale=1]{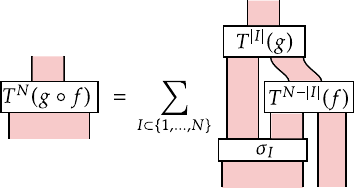}}}\;.
    \end{equation}
\end{lem}
\begin{proof}
    We prove this by induction on $N$. First, notice that the base case $N=0$ is trivially satisfied. Now assume that \eqref{eqn:T-on-comp} holds for $N=K$. We then have that
    \begin{align*}
        &\vcenter{\hbox{\includegraphics[scale=1]{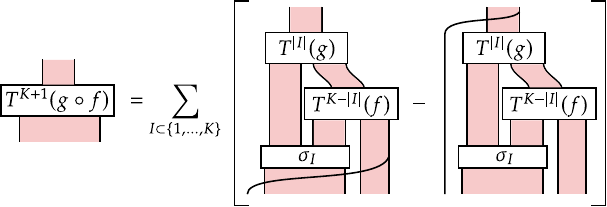}}}\\
        &\vcenter{\hbox{\includegraphics[scale=1]{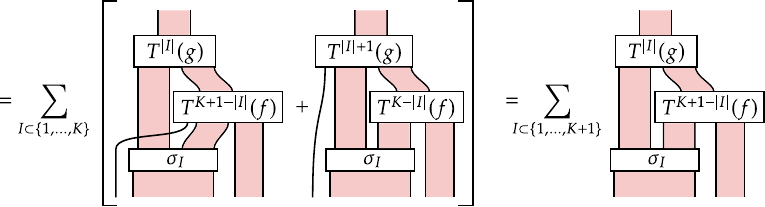}}}\;,
    \end{align*}
    showing that \eqref{eqn:T-on-comp} holds for $N=K+1$.
\end{proof}

\begin{prop}\label{prop:comm-of-diffops}
    Let $\Delta \in \mathrm{D}_k(a)$ and $\nabla \in \mathrm{D}_l(a)$. Then $\Delta \circ \nabla \in \mathrm{D}_{k+l}(a)$. Thus, $\mathrm{D}(a)$ is a filtered subalgebra of $\mathcal{C}(a,a)$.
    In addition, $[\Delta,\nabla] = \Delta \circ \nabla - \nabla \circ \Delta \in \mathrm{D}_{k+l-1}(a)$.\footnote{To make sense of the case $k=l=0$, we understand $\mathrm{D}_{-1}(a)$ to be $0$.}
\end{prop}
\begin{proof}
    The first claim is an immediate consequence of Lemma \ref{lem:T-on-comp}.
    For the second, using that $T^{k+1}(\Delta)$ and $T^{l+1}(\nabla)$ vanish, Lemma \ref{lem:T-on-comp} shows that $T^{k+l}([\Delta,\nabla])$ is equal to
    \begin{equation}\label{eqn:T-on-commutator}
        \vcenter{\hbox{\includegraphics[scale=1]{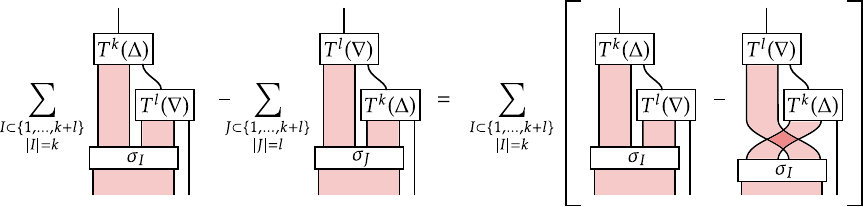}}}\;.
    \end{equation}
    Notice that
    \begin{equation*}
        \vcenter{\hbox{\includegraphics[scale=1]{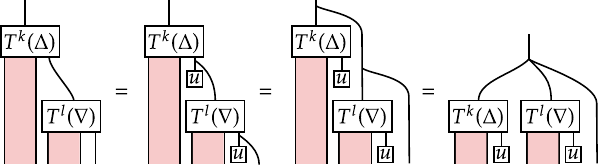}}}\;,
    \end{equation*}
    where $u:1\to a$ is the unit and, in the second equality, we used that $T(T^k(\Delta))$ and $T(T^l(\nabla))$ vanish. But then the same reasoning shows that 
    \begin{equation*}
        \vcenter{\hbox{\includegraphics[scale=1]{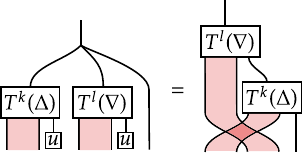}}}\;
    \end{equation*}
    and hence that \eqref{eqn:T-on-commutator} vanishes.
\end{proof}

To understand better what it means for a map $\Delta:a \to a$ to be a differential operator, we first understand what $T^N(\Delta)$ is explicitly.
\begin{lem}\label{lem:T^(f)}
    Let $\Delta \in \mathcal{C}(a,a)$. For all $N \geq 0$,
    \begin{equation}\label{eqn:T^N(f)}
        \vcenter{\hbox{\includegraphics[scale=1]{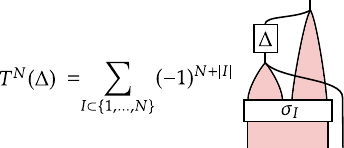}}}\;.
    \end{equation}
\end{lem}
\begin{proof}
    We use induction on $N$. The case $N=0$ holds trivially. Assume that \eqref{eqn:T^N(f)} holds for $N=K$. Then we have
    \begin{align*}
        \vcenter{\hbox{\includegraphics[scale=1]{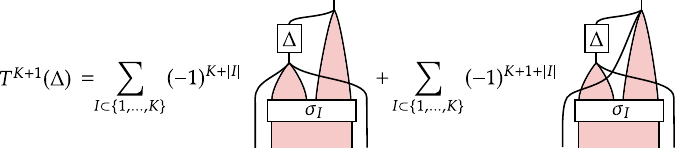}}}\\
        \vcenter{\hbox{\includegraphics[scale=1]{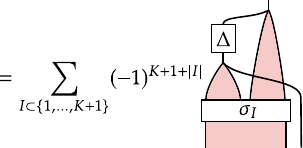}}}\;.
    \end{align*}
\end{proof}
\begin{lem}\label{lem:precompose-by-u}
    Let $f:a^{\otimes m} \to a^{\otimes n}$ be a map. Then $T(f) = 0$ if and only if $T(f) \circ (\mathrm{id}_{a^{\otimes m}} \otimes u) = 0$.
\end{lem}
\begin{proof}
    The forward implication is immediate. For the backward implication, note that $T(f) \circ (\mathrm{id}_{a^{\otimes m}} \otimes u) = 0$ is equivalent to
    \begin{equation*}
        \vcenter{\hbox{\includegraphics[scale=1]{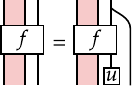}}}\;.
    \end{equation*}
    Using this twice, we have
    \begin{equation*}
        \vcenter{\hbox{\includegraphics[scale=1]{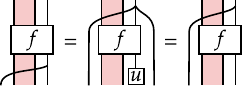}}}\;,
    \end{equation*}
    which is precisely $T(f)=0$.
\end{proof}

Combining Lemmas $\ref{lem:T^(f)}$ and \ref{lem:precompose-by-u}, we have the following explicit description of differential operators:\footnote{In the ordinary setting, this has been observed in Equation (18) of \cite{markl2001loop}, for example.}
\begin{prop}\label{prop:diffop-condition}
    A map $\Delta:a\to a$ is a differential operator of order $k$ if and only if
    \begin{equation*}
        \vcenter{\hbox{\includegraphics[scale=1]{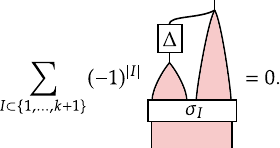}}}
    \end{equation*}
\end{prop}
In particular, Proposition \ref{prop:diffop-condition} says that the value of $\Delta \in \mathrm{D}_k(a)$ on a $(k+1)$-fold product (i.e.\ the term with $I = \{1,\dots,k+1\}$) is determined by its values on $(\leq k)$-fold products. Inductively, its value on an $N$-fold product for any $N \geq k+1$ is also determined. More explicitly:
\begin{prop}\label{prop:explicit-diffop}
    Let $\Delta \in \mathrm{D}_{k}(a)$ be a differential operator of order $k$, and let $N \geq k+1$. Then
    \begin{equation*}
        \vcenter{\hbox{\includegraphics[scale=1]{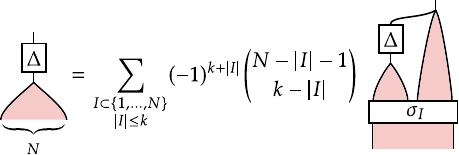}}}\;.
    \end{equation*}
\end{prop}
\begin{proof}
    See Appendix \ref{Appendix A}.
\end{proof}

Let us make the observation that $\Delta \in \mathrm{D}_{k}(a)$ is `determined on $(\leq k)$-fold products' helpful. For the next two results we assume that $\mathcal{C}$ is $\otimes$-cocomplete, so that free commutative monoids in $\mathcal{C}$ exist and can be modelled via Construction \ref{constr:free-com-mon}. Denote by $S^{\leq k}x$ the object $\bigsqcup_{n=0}^k S^{n}x$.
\begin{prop}\label{prop:generators-diffops}
    Suppose $\mathcal{C}$ is $\otimes$-cocomplete. Let $x \to a$ be a map in $\mathcal{C}$ such that the canonical map of monoids $Sx \to a$ is an epimorphism in $\mathcal{C}$ (we can say that $x \to a$ \textit{generates} $a$ as a commutative monoid). Given any map $\phi:S^{\leq k}x \to a$, there exists at most one $\Delta \in \mathrm{D}_{k}(a)$ such that the following diagram commutes:
    \begin{equation}\label{eqn:comm-diag}
        \begin{tikzcd}
    	{S^{\leq k}x} & a \\
    	& a
    	\arrow["{\mathrm{can}}", from=1-1, to=1-2]
    	\arrow["\phi"', from=1-1, to=2-2]
    	\arrow["\Delta", dashed, from=1-2, to=2-2]
    \end{tikzcd}.
    \end{equation}
   Equivalently, $\mathrm{can}^*:\mathrm{D}_{k}(a) \to \mathcal{C}(S^{\leq k}x,a)$ is injective.
\end{prop}
\begin{proof}
   Suppose that $\Delta$ and $\Delta'$ are two such differential operators. We will show that they are equal. Denote the map $x \to a$ by $f$. Recall that the canonical map $Sx \to a$ is equivalently the collection of $S_n$-coequalising maps $x^{\otimes n} \to a$ (for $n \geq 0$) given by the composites
   \begin{equation*}
       x^{\otimes n} \stackrel{f^{\otimes n}}{\to} a^{\otimes n} \stackrel{\mu}{\to} a
       =
       \;\vcenter{\hbox{\includegraphics[scale=1]{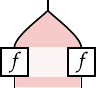}}}\;.
   \end{equation*}
   Similarly, for $n \leq k$, let $\varphi_n: x^{\otimes n} \to a$ denote the $S_n$-coequalising lifts of $\phi$. That $\Delta$ and $\Delta'$ make \eqref{eqn:comm-diag} commute says that for all $n \leq k$ we have
   \begin{equation*}
       \Delta \circ \mu \circ f^{\otimes n} = \varphi_n = \Delta' \circ \mu \circ f^{\otimes n}.
   \end{equation*}
   Meanwhile, $\Delta$ and $\Delta'$ are both order $k$ differential operators, so Proposition \ref{prop:explicit-diffop} shows that for all $N \geq k+1$ we have
   \begin{equation*}
       \vcenter{\hbox{\includegraphics[scale=1]{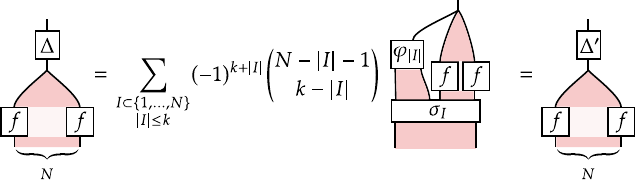}}}\;.
   \end{equation*}
   Thus, $\Delta \circ \mu \circ f^{\otimes n} = \Delta' \circ \mu \circ f^{\otimes n}$ for \textit{all} $n \geq 0$. Equivalently,
   \begin{equation*}
       Sx \stackrel{\mathrm{can}}{\to} a \stackrel{\Delta}{\to} a = Sx \stackrel{\mathrm{can}}{\to} a \stackrel{\Delta'}{\to} a.
   \end{equation*}
   But the canonical map $Sx \to a$ is epic, so $\Delta = \Delta'$.
\end{proof}

\begin{thm}\label{thm:!existence-of-diffops}
    Suppose $\mathcal{C}$ is $\otimes$-cocomplete. Let $x \in \mathcal{C}$. For all maps $\phi:S^{\leq k} x \to Sx$, there exists a unique $\Delta \in \mathrm{D}_{k}(Sx)$ such that the following diagram commutes:
    \begin{equation*}
        \begin{tikzcd}
            {S^{\leq k}x} & Sx \\
            & Sx
            \arrow["{\mathrm{can}}", from=1-1, to=1-2]
            \arrow["\phi"', from=1-1, to=2-2]
            \arrow["\Delta", dashed, from=1-2, to=2-2]
        \end{tikzcd}.
    \end{equation*}
    Equivalently, restriction $\mathrm{D}_{k}(Sx) \isoto \mathcal{C}(S^{\leq k}x,Sx)$ is an isomorphism of $R$-modules.
\end{thm}
\begin{proof}
    We have uniqueness by Proposition \ref{prop:generators-diffops}, applied to the natural map $x\cong S^1 x \to Sx$ (whose associated map $Sx \to Sx$ is an isomorphism, hence epic). It remains to show existence. Defining $\Delta$ amounts to defining maps $S^nx \to Sx$ for each $n$. Define $\Delta$ to agree with $\phi$ on $S^n x$ for $n \leq k$. Meanwhile, as suggested by Proposition \ref{prop:explicit-diffop}, for $N \geq k+1$ we define
    \begin{equation}\label{eqn:002}
        \vcenter{\hbox{\includegraphics[scale=1]{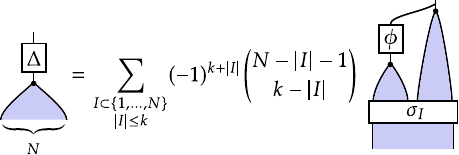}}}\;,
    \end{equation}
    where each of the $N$ strands at the bottom denotes the object $x \cong S^1 x$, and the dots denote projections $x^{\otimes n} \twoheadrightarrow S^n x$ (which we can simply think of as multiplication $(S^1 x)^{\otimes n} \to S^n x$). Notice that we can (and will) use \eqref{eqn:002} for \textit{all} $N \geq 0$, provided that $\binom{\alpha}{\beta} := 0$ for $\beta > \alpha \geq 0$, and that $\binom{-1}{\beta} := (-1)^{\beta}$ for $\beta \geq 0$. 
    
    It remains to show that $\Delta$ is a differential operator of order $k$, i.e.\ that $T^{k+1}(\Delta) \circ (\mathrm{id} \otimes u) = 0$ as a map $(Sx)^{\otimes {k+1}} \to Sx$. Equivalently, that the components $S^{N_1}x \otimes \cdots \otimes S^{N_{k+1}}x \to Sx$ of this map are 0 for all $N_1,\dots,N_{k+1} \geq 0$, which by Proposition \ref{prop:diffop-condition} says that
    \begin{equation}\label{eqn:003}
        \vcenter{\hbox{\includegraphics[scale=1]{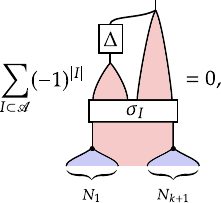}}}\;
    \end{equation}
    where we have denoted $\{1,\dots,k+1\}$ by $\mathscr{A}$. Before showing that \eqref{eqn:003} holds, we introduce some notation to keep track of indexing sets: For each $i \in \mathscr{A}$, let $\mathscr{B}_i := \{1,\dots,N_i\}$, and, for each subset $I \subset \mathscr{A}$, define $\mathscr{B}_I := \bigsqcup_{i\in I}\mathscr{B}_i$. We write $\mathscr{B} := \mathscr{B}_{\mathscr{A}}$ for the full indexing set (with the natural total order). Notice that the left hand side of \eqref{eqn:003} can be rewritten as
    \begin{equation*}
        \vcenter{\hbox{\includegraphics[scale=1]{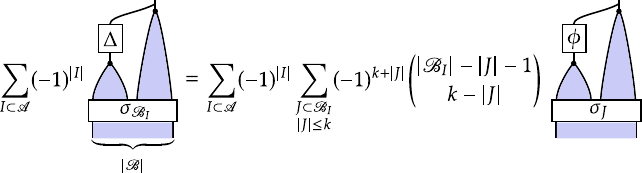}}}\;.
    \end{equation*}
    Changing the order of summation, this is nothing but
    \begin{equation*}
        \vcenter{\hbox{\includegraphics[scale=1]{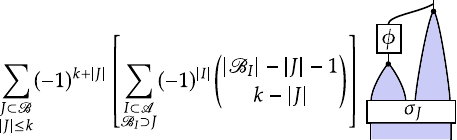}}}\;.
    \end{equation*}
    The term inside the square brackets vanishes for each $J$ --- see Lemma \ref{lem:for-!existence-of-diffops} in Appendix \ref{Appendix A} --- showing that \eqref{eqn:003} holds.
\end{proof}

A \textbf{derivation} is a 1st order differential operator that kills the unit. The $R$-module of derivations on $a$ is denoted $\mathrm{Der}(a)$ and is a Lie subalgebra of $\mathcal{C}(a,a)$ by Proposition \ref{prop:comm-of-diffops}. When $\mathcal{C}$ is $\otimes$-cocomplete, Theorem \ref{thm:!existence-of-diffops} shows that $\mathrm{Der}(Sx) \cong \mathcal{C}(x,Sx)$.

\begin{prop}\label{prop:diffops-to-diffops}
    Let $F:\mathcal{C} \to \mathcal{D}$ be a linear lax symmetric monoidal functor between linear SMCs, and let $a \in \mathcal{C}$ be a commutative monoid. The map $\mathcal{C}(a,a) \to \mathcal{D}(F(a),F(a))$ restricts to a filtered algebra morphism
    \begin{equation*}
        \mathrm{D}(a) \to \mathrm{D}(F(a)),
    \end{equation*}
    where $F(a)$ has the canonical commutative monoid structure.
\end{prop}
\begin{proof}
    It suffices to show that $f \in \mathrm{D}_{k}(a) \implies F(f) \in \mathrm{D}_{k}(F(a))$ for all $k \geq 0$. Writing $\eta:1_{\mathcal{D}} \to F(1_{\mathcal{C}})$ and $\left(J_{x,y}:F(x) \otimes F(y) \to F(x \otimes y)\right)_{x,y \in \mathcal{C}}$ for the structure maps of $F$, recall that the multiplication and unit on $F(a)$ are given by $F(\mu) \circ J_{a,a}$ and $F(u) \circ \eta$, where $\mu$ and $u$ denote the multiplication and unit on $a$. It can then be shown that the following diagram commutes for all $m \geq 1$:
    \begin{equation*}
        \begin{tikzcd}
        	{\mathcal{C}(a^{\otimes m},a)} & {\mathcal{D}(F(a^{\otimes m}),F(a))} & {\mathcal{D}(F(a)^{\otimes m},F(a))} \\
        	{\mathcal{C}(a^{\otimes (m+1)},a)} & {\mathcal{D}(F(a^{\otimes (m+1)}),F(a))} & {\mathcal{D}(F(a)^{\otimes (m+1)},F(a))}
        	\arrow["F", from=1-1, to=1-2]
        	\arrow["T"', from=1-1, to=2-1]
        	\arrow[from=1-2, to=1-3]
        	\arrow["T", from=1-3, to=2-3]
        	\arrow["F", from=2-1, to=2-2]
        	\arrow[from=2-2, to=2-3]
        \end{tikzcd},
    \end{equation*}
    where the rightmost horizontal maps are pullbacks along the canonical maps $F(a)^{\otimes m} \to F(a^{\otimes m})$ induced by $J$. For $m=1$ this canonical map is identity, so we obtain a commutative diagram
    \begin{equation*}
        \begin{tikzcd}
        	{\mathcal{C}(a,a)} & {\mathcal{C}(a^{\otimes 2},a)} & {\mathcal{C}(a^{\otimes 3},a)} & \cdots \\
        	{\mathcal{D}(F(a),F(a))} & {\mathcal{D}(F(a)^{\otimes 2},F(a))} & {\mathcal{D}(F(a)^{\otimes 3},F(a))} & \cdots
        	\arrow["T", from=1-1, to=1-2]
        	\arrow["F"', from=1-1, to=2-1]
        	\arrow["T", from=1-2, to=1-3]
        	\arrow[from=1-2, to=2-2]
        	\arrow[from=1-3, to=1-4]
        	\arrow[from=1-3, to=2-3]
        	\arrow["T", from=2-1, to=2-2]
        	\arrow["T", from=2-2, to=2-3]
        	\arrow[from=2-3, to=2-4]
        \end{tikzcd},
    \end{equation*}
    from which the result follows.
\end{proof}

\begin{expl}
    The linear functor $\mathcal{C}(1,-):\mathcal{C} \to \mathrm{mod}_R$ taking objects to their $1$-points is canonically lax symmetric monoidal, so the above proposition applies here.
\end{expl}

\subsection{Bi-differential operators}\label{sec:bi-differential operators}
The language of \S\ref{sec:Diffops} can be easily extended to allow one to consider maps $a \otimes a \to a$ that are `differential operators in each variable'. Namely, one considers two families of $R$-linear maps:
\begin{align*}
    &T_1:\mathcal{C}(a^{\otimes m_1} \otimes a^{\otimes m_2},a^{\otimes n}) \to \mathcal{C}(a^{\otimes (m_1+1)} \otimes a^{\otimes m_2},a^{\otimes n}); \; \text{and}\\
    &T_2:\mathcal{C}(a^{\otimes m_1} \otimes a^{\otimes m_2},a^{\otimes n}) \to \mathcal{C}(a^{\otimes m_1} \otimes a^{\otimes (m_2 + 1)},a^{\otimes n}),
\end{align*}
declaring that $\Delta \in \mathcal{C}(a \otimes a,a)$ is a \textbf{bi-differential operator} of \textbf{order} $(k_1,k_2)$ if $T_1^{k_1 + 1}(\Delta) = 0$ and $T_2^{k_2 + 1}(\Delta) = 0$. Of course, one has the notion of a \textbf{bi-derivation}, which is an order $(1,1)$ bi-differential operator $\Delta$ such that $\Delta \circ (\mathrm{id}_a \otimes u) = \Delta \circ (u \otimes \mathrm{id}_a) = 0$. When $\mathcal{C}$ is  is $\otimes$-cocomplete, bi-differential operators on a free commutative monoid $Sx$ are easy: those with order $(k_1,k_2)$ are in bijection with $\mathcal{C}(S^{\leq k_1}x \otimes S^{\leq k_2}x , Sx)$.

\begin{defn}
    A \textbf{Poisson algebra in $\mathcal{C}$} is a commutative monoid $a$ in $\mathcal{C}$ together with a bi-derivation $\{-,-\}:a \otimes a \to a$ (called a \textbf{Poisson bracket}) that is simultaneously a Lie bracket, meaning that it satisfies the following two properties:
    \begin{enumerate}
        \item \textit{Antisymmetry}. $\{-,-\} \circ \sigma_{a,a} = -\{-,-\}$; and
        \item\label{itm:007} \textit{Jacobi identity}. $\{-,\{-,-\}\} \circ \mathfrak{S}_3 = 0$, where $\{-,\{-,-\}\}$ is shorthand for $\{-,-\} \circ (\mathrm{id}_a \otimes \{-,-\})$ and $\mathfrak{S}_3$ is the cyclic symmetriser
        \begin{equation*}
            \includegraphics[scale=1]{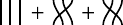}\;.
        \end{equation*}
    \end{enumerate}
\end{defn}
It is easy to check that if $\{-,-\}$ is an antisymmetric bi-derivation, then the map $\{-,\{-,-\}\} \circ \mathfrak{S}_3: a\otimes a\otimes a \to a$ is a \textit{tri-derivation}. In particular, if $\mathcal{C}$ is $\otimes$-cocomplete and $a = Sx$, then to show Property \ref{itm:007} it suffices to show that $\{-,\{-,-\}\} \circ \mathfrak{S}_3$ vanishes on $x \otimes x \otimes x$, by the appropriate analogue of Theorem \ref{thm:!existence-of-diffops}. This observation will be useful for us when we construct Poisson algebras later.

Finally, by an analogue of Proposition \ref{prop:diffops-to-diffops}, one sees that linear lax symmetric monoidal functors take Poisson algebras to Poisson algebras.

\subsection{Linear SMCs with parity structure; BV algebras}
With the goal of discussing \textit{BV algebras} in $\mathcal{C}$, we need $\mathcal{C}$ to admit a workable theory of even and odd maps. So far we have taken $\mathcal{C}$ to be an $R$-linear symmetric monoidal category, perhaps with some nice properties. Our first guess would be to simply replace the enriching category $\mathrm{mod}_R$ with the category $\mathrm{smod}_R$ of super $R$-modules, in which case $\mathcal{C}$ would be called a \textit{symmetric monoidal supercategory}. But this is not enough --- we also need a notion of \textit{parity-inversion}, in this context arising as the $\mathrm{smod}_R$-functor (i.e.\ \textit{superfunctor}) $\pi \otimes - : \mathcal{C} \to \mathcal{C}$ for $\pi \in \mathcal{C}$ an object equipped with an odd isomorphism $\pi \isoto 1$. With such data, $\mathcal{C}$ would be a \textit{symmetric monoidal $\Pi$-supercategory}. 

In \cite{brundan2017monoidal} it was shown that the $2$-category of monoidal $\Pi$-supercategories is equivalent to the $2$-category of \textit{monoidal $\Pi$-categories}, which are $R$-linear monoidal categories $\mathcal{C}$ together with an object $(\pi,\beta) \in Z(\mathcal{C})$ in the Drinfeld centre satisfying certain axioms. In particular, $\mathrm{smod}_R$-enrichment comes for free. As such, we use the following definition, which is a specialisation of the definition of monoidal $\Pi$-categories in \cite{brundan2017monoidal}.
\begin{defn}
    Given an $R$-linear symmetric monoidal category $\mathcal{C}$, a \textbf{parity structure on $\mathcal{C}$} is a pair $(\pi,\xi)$ consisting of an object $\pi \in \mathcal{C}$ such that $\sigma_{\pi,\pi} = -\mathrm{id}_{\pi \otimes \pi}$, together with an isomorphism $\xi: \pi \otimes \pi \isoto 1$. In this context the functor $\pi \otimes - : \mathcal{C} \to \mathcal{C}$ is denoted $\Pi$. Equipped with a choice of parity structure, $\mathcal{C}$ is called a \textbf{symmetric monoidal $\Pi$-category}.
\end{defn}
\begin{rem}\label{rem:dualisable?}
    The isomorphism $\pi \otimes \pi \isoto 1$ exhibits $\pi$ as a self-inverse object in $\mathcal{C}$, equivalently, as a self-inverse equivalence in the delooping $B\mathcal{C}$. Since equivalences in $2$-categories can be promoted to \textit{adjoint} equivalences by redefining the invertible structure 2-morphisms, it follows that $\pi$ is self-dual in $\mathcal{C}$, even though $(\xi,\xi^{-1})$ may not be duality data.
\end{rem}
One can talk about structure-preserving functors between linear SMCs with parity structure.
\begin{defn}
    Let $(\mathcal{C},\pi,\xi)$ and $(\mathcal{D},\tau,\zeta)$ be symmetric monoidal $\Pi$-categories. A \textbf{lax symmetric monoidal $\Pi$-functor}\footnote{The adjective `lax' is dropped in \cite{brundan2017monoidal}, where it is understood implicitly.} $\mathcal{C} \to \mathcal{D}$ is a linear lax symmetric monoidal functor $F$ together with a map $\omega:\tau \to F(\pi)$ in $\mathcal{D}$ such that the following diagram commutes:
    \begin{equation*}
        \begin{tikzcd}
        	{\tau \otimes \tau} & 1 & {F(1)} \\
        	{F(\pi)\otimes F(\pi)} && {F(\pi \otimes \pi)}
        	\arrow["\zeta", from=1-1, to=1-2]
        	\arrow["{\omega \otimes \omega}"', from=1-1, to=2-1]
        	\arrow["\eta", from=1-2, to=1-3]
        	\arrow["{J_{\pi,\pi}}", from=2-1, to=2-3]
        	\arrow["{F(\xi)}"', from=2-3, to=1-3]
        \end{tikzcd},
    \end{equation*}
    where $\eta$ and $J$ are the coherence maps of $F$. Together with composition $(G,\omega_G) \circ (F, \omega_F) := (G \circ F, G(\omega_F) \circ \omega_G)$ and the obvious identity maps, we get a category $\mathrm{SymMon\Pi\text{-}Cat}_R$.
\end{defn}

\begin{constr}\label{constr:supercat}
    Let $(\mathcal{C},\pi,\xi) \in \mathrm{SymMon\Pi\text{-}Cat}_R$. Given the parity structure, one can define a notion of odd morphisms, obtaining a symmetric monoidal supercategory $\underline{\mathcal{C}} \in \mathrm{SymMonCat}_{\mathrm{smod}_R}$ as follows: 
    \begin{itemize}
        \item $\mathrm{ob}\,\underline{\mathcal{C}} := \mathrm{ob}\,\mathcal{C}$.
        
        \item $\underline{\mathcal{C}}(x,y)  \in \mathrm{smod}_R$ is defined to have even part $\mathcal{C}(x,y)$ and odd part $\mathcal{C}(\pi \otimes x,y)$. That is, a (homogeneous) map $f : x \to y$ in $\underline{\mathcal{C}}$ is a map $\pi^{|f|} \otimes x \to y$ in $\mathcal{C}$, where $|f|$ denotes the parity of $f$.
        
        \item Composition in $\underline{\mathcal{C}}$ is given by the following maps in $\mathrm{smod}_R$:
        \begin{equation*}
            \underline{\mathcal{C}}(y,z) \otimes \underline{\mathcal{C}}(x,y) \to \underline{\mathcal{C}}(x,z);\quad
            g \otimes f \mapsto \left[\pi^{|g| + |f|} \otimes x \cong \pi^{|g|}\otimes \pi^{|f|}\otimes x \stackrel{\mathrm{id} \otimes f}{\longrightarrow} \pi^{|g|} \otimes y \stackrel{g}{\to} z \right],
        \end{equation*}
        where the isomorphism uses $\xi$ (or $\xi^{-1}$) --- from now on we leave it implicit. \textit{Note: showing associativity of composition uses $\sigma_{\pi,\pi} = - \mathrm{id}_{\pi \otimes \pi}$.}

        \item The identity map on $x$ in $\underline{\mathcal{C}}$ is simply the identity map in $\mathcal{C}$, so there is no ambiguity writing `$\mathrm{id}_x$'.
        \item $\otimes_{\underline{\mathcal{C}}}$ is the $\mathrm{smod}_R$-functor $\underline{\mathcal{C}} \otimes \underline{\mathcal{C}} \to \underline{\mathcal{C}}$ that on objects agrees with the monoidal product of $\mathcal{C}$, while for morphism spaces it is given by the following maps in $\mathrm{smod}_R$:
        \begin{align*}
            \underline{\mathcal{C}}(x,y) \otimes \underline{\mathcal{C}}(x',y') &\to \underline{\mathcal{C}}(x \otimes x', y \otimes y');\\
            f \otimes f' &\mapsto \left[ \pi^{|f|+|f'|} \otimes x \otimes x' \isoto \pi^{|f|} \otimes x \otimes \pi^{|f'|} \otimes x' \stackrel{f \otimes f'}{\longrightarrow} y \otimes y'\right],
        \end{align*}
        where the isomorphism uses the symmetry $\sigma$ of $\mathcal{C}$.
        \textit{Note: $\sigma_{\pi,\pi} = - \mathrm{id}_{\pi \otimes \pi}$ is crucial for $\otimes_{\underline{\mathcal{C}}}$ to respect composition!}
        
        \item The unit object in $\underline{\mathcal{C}}$ is simply the unit object in $\mathcal{C}$ --- so there is no ambiguity writing `$1$'.

        \item The symmetry on $\underline{\mathcal{C}}$ is essentially the symmetry $\sigma$ on $\mathcal{C}$: it is the $\mathrm{smod}_R$-natural transformation whose components are the even maps $\sigma_{x,y} \in \mathcal{C}(x \otimes y, y \otimes x)$.
    \end{itemize}
    Furthermore, given a functor $(F.\omega):(\mathcal{C},\pi,\xi) \to (\mathcal{D},\tau,\zeta)$ in $\mathrm{SymMon\Pi\text{-}Cat}_R$, we get a lax symmetric monoidal $\mathrm{smod}_R$-functor $\underline{F}:\underline{\mathcal{C}} \to \underline{\mathcal{D}}$:
    
    \begin{itemize}
        \item $\underline{F}(x) = F(x)$.
        \item $\underline{\mathcal{C}}(x,y) \to \underline{\mathcal{D}}(\underline{F}(x),\underline{F}(y))$ is just $F$ on even maps, while on odd maps it is
        \begin{align*}
            \mathcal{C}(\pi \otimes x,y) \to \mathcal{D}(\tau \otimes F(x),F(y)); \quad 
            f \mapsto \left[\tau \otimes F(x) \stackrel{\omega \otimes \mathrm{id}}{\longrightarrow} F(\pi) \otimes F(x) \stackrel{J_{\pi,x}}{\longrightarrow} F(\pi \otimes x) \stackrel{F(f)}{\to} F(y) \right].
        \end{align*}
        \item The coherence maps for $\underline{F}$ are simply those for $F$, viewed as even maps.
    \end{itemize}
    All considered, this gives a functor 
    \begin{equation*}
        (\,\underline{\;\;}\,):\mathrm{SymMon\Pi\text{-}Cat}_R \to \mathrm{SymMonCat}_{\mathrm{smod}_R}.
    \end{equation*}
\end{constr}
We will later need symmetric monoidal $\Pi$-categories that are $\otimes$-cocomplete. The following construction shows that they are plentiful.
\begin{constr}\label{constr:addmon-to-pimon}
    Given an \textit{additive} $R$-linear symmetric monoidal category $\mathcal{C}$ (in particular, $\otimes$ is an additive functor), one can construct its category of `super objects' $\mathcal{SC} \in \mathrm{SymMon\Pi\text{-}Cat}_R$ in the following way:
    \begin{itemize}
        \item As an $R$-linear category, $\mathcal{SC}$ has objects that are pairs $x = (x_0,x_1)$ of objects of $\mathcal{C}$, and mapping spaces are $\mathcal{SC}(x,y) = \mathcal{C}(x_0,y_0) \oplus \mathcal{C}(x_1,y_1)$, with component-wise composition.

        \item $\otimes_{\mathcal{SC}}:\mathcal{SC} \times \mathcal{SC} \to \mathcal{SC}$ is defined on objects by
        \begin{equation*}
            x\otimes_{\mathcal{SC}} y = ((x_0 \otimes y_0) \oplus (x_1 \otimes y_1), (x_0 \otimes y_1) \oplus (x_1 \otimes y_0)),
        \end{equation*}
        and on maps in the obvious way.

        \item The unit object in $\mathcal{SC}$ is $(1,0)$.

        \item The components of the symmetry on $\mathcal{SC}$ are the linear maps $(x \otimes_{\mathcal{SC}} y \to y \otimes_\mathcal{SC} x)_{x,y \in \mathcal{SC}}$ given by 
        \begin{equation}\label{eqn:018}
            (-1)^{ij} \sigma_{x_i, y_j} : x_i \otimes y_j \to y_j \otimes x_i, \quad \text{for $i,j = 0,1$,}
        \end{equation}
        where $\sigma$ denotes the symmetry on $\mathcal{C}$.
        \item The parity structure is $\pi = (0,1)$, with $\xi: \pi \otimes_{\mathcal{SC}} \pi \isoto 1_{\mathcal{SC}}$ the obvious isomorphism.
    \end{itemize}
    Notice that $\mathcal{SC}$ is itself additive. Furthermore, if $\mathcal{C}$ has \textit{all} colimits with $\otimes$ colimit-preserving, then $\mathcal{SC}$ has all colimits with $\otimes_{\mathcal{SC}}$ colimit-preserving too. Similarly, $\mathcal{SC}$ is closed if $\mathcal{C}$ is.
    
    Given an additive $R$-linear lax symmetric monoidal functor $F:\mathcal{C} \to \mathcal{D}$, we get a lax symmetric monoidal $\Pi$-functor $\mathcal{S}F:\mathcal{SC} \to \mathcal{SD}$ in the natural way. All considered, have a functor
    \begin{equation*}
        \mathcal{S}:\mathrm{AddSymMonCat}_R \to \mathrm{SymMon\Pi\text{-}Cat}_R.
    \end{equation*}
\end{constr}
\begin{rem}
    The enriching category $\mathrm{smod}_R$ (viewed as a $\Pi$-category with parity structure $\pi = R_{\text{(odd)}}$) arises in this way: $\mathrm{smod}_R \simeq \mathcal{S}\mathrm{mod}_R$.
\end{rem}

For the rest of this subsection, let $\mathcal{C} \in \mathrm{SymMon\Pi\text{-}Cat}_R$ have parity structure $\pi$, and let $a \in \mathcal{C}$ be a commutative monoid. Via the category $\underline{\mathcal{C}}$, which permits us to speak of odd maps as well as even ones, we can speak of even and odd differential operators: in analogy with before, for all $m,n \geq 1$, consider the $\mathrm{smod}_R$-maps $\underline{\mathcal{C}}(a^{\otimes m},a^{\otimes n}) \to \underline{\mathcal{C}}(a^{\otimes (m+1)},a^{\otimes n})$ given by
\begin{equation*}
    \vcenter{\hbox{\includegraphics[scale=1]{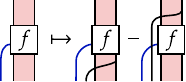}}}\;,
\end{equation*}
where the blue strand depicts $\pi^{|f|}$ (here $f$ is taken to be homogeneous). This gives a sequence of maps
\begin{equation}\label{eqn:application-of-t-graded}
    \underline{\mathcal{C}}(a,a) \to \underline{\mathcal{C}}(a \otimes a ,a) \to \underline{\mathcal{C}}(a \otimes a \otimes a,a) \to \cdots
\end{equation}
which, just as in Definition \ref{defn:diffop}, allows us to define a sub-supermodule $\underline{\mathrm{D}}(a) \subset \underline{\mathcal{C}}(a,a)$ of differential operators in the graded setting, with filtration by order
\begin{equation*}
    \underline{\mathrm{D}}_0(a) \subset \underline{\mathrm{D}}_1(a) \subset \underline{\mathrm{D}}_2(a) \subset \dots \subset \underline{\mathrm{D}}(a).
\end{equation*}
The results of \S\ref{sec:Diffops} hold in this setting, modified appropriately. For example:
\begin{prop}\label{prop:graded-comm-of-diffops}
    Let $\Delta \in \underline{\mathrm{D}}_k(a)$ and $\nabla \in \underline{\mathrm{D}}_l(a)$. Then $\Delta \circ \nabla \in \underline{\mathrm{D}}_{k+l}(a)$. Thus, $\underline{\mathrm{D}}(a)$ is a filtered sub-superalgebra of $\underline{\mathcal{C}}(a,a)$. Furthermore, $[\Delta,\nabla] \in \underline{\mathrm{D}}_{k+l-1}(a)$, where $[-,-]$ denotes the \textit{graded} commutator.
\end{prop}

\begin{thm}\label{thm:!existence-of-diffops-graded}
    Suppose that the symmetric monoidal $\Pi$-category $\mathcal{C}$ is $\otimes$-cocomplete. 
    \begin{enumerate}
        \item\label{itm:003} Let $x \to a$ be a map in $\mathcal{C}$ such that the canonical map of monoids $Sx \to a$ is epic in $\mathcal{C}$ (and hence epic in $\underline{\mathcal{C}}$). Then $\mathrm{can}^*:\underline{\mathrm{D}}_k(a) \to \underline{\mathcal{C}}(S^{\leq k}x,a)$ is injective.

        \item\label{itm:004} For all $x \in \mathcal{C}$, restriction gives an isomorphism of supermodules $\underline{\mathrm{D}}_k(Sx) \cong \underline{\mathcal{C}}(S^{\leq k}x,Sx)$.
    \end{enumerate}
\end{thm}
\begin{proof}
    To show Claim \ref{itm:003}, suppose that $\Delta,\Delta' \in \underline{\mathrm{D}}_k(a)$ are two differential operators such that $\Delta \circ \mathrm{can} = \Delta' \circ \mathrm{can}$ in $\underline{\mathcal{C}}(S^{\leq k}x,a)$. In terms of even and odd parts, this means that 
    \begin{enumerate}
        \item\label{itm:001} $\Delta_0 \circ \mathrm{can} = \Delta'_0 \circ \mathrm{can}$ in $\mathcal{C}(S^{\leq k}x,a)$; and
        \item\label{itm:002} $\Delta_1 \circ (\mathrm{id}_\pi \otimes \mathrm{can}) = \Delta'_1 \circ (\mathrm{id}_\pi \otimes \mathrm{can})$ in $\mathcal{C}(\pi \otimes S^{\leq k}x,a)$.
    \end{enumerate}
    Item \ref{itm:001} implies that $\Delta_0 = \Delta'_0$, by Proposition \ref{prop:generators-diffops}. Meanwhile, arguing as in the proof of Proposition \ref{prop:generators-diffops} (but with an additional strand depicting $\pi$ appropriately added), we see that Item \ref{itm:002} gives $\Delta_1 = \Delta'_1$.

    For Claim \ref{itm:004}, isomorphism of the even parts is just Theorem \ref{thm:!existence-of-diffops}. For the odd parts, one argues as in its proof (but again with an appropriately added strand).
\end{proof}

Spelling this out, if $\phi \in \underline{\mathcal{C}}(S^{\leq k}x,Sx)$ is homogeneous, then the unique differential operator $\Delta \in \underline{\mathrm{D}}_k(Sx)$ extending $\phi$ is determined by
\begin{equation}\label{eqn:explicit-graded-diffop}
    \vcenter{\hbox{\includegraphics[scale=1]{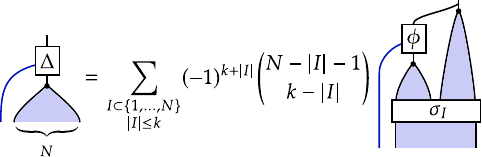}}}\;,
\end{equation}
for all $N \geq k+1$, where each of the $N$ strands at the bottom denotes the object $x \cong S^{1}x$.

\begin{prop}\label{prop:diffops-to-diffops-graded}
    Let $(F,\omega) : (\mathcal{C},\pi,\xi) \to (\mathcal{D},\tau,\zeta)$ be a lax symmetric monoidal $\Pi$-functor, and let $a \in \mathcal{C}$ be a commutative monoid. With respect to the commutative monoid structure on $F(a)$, the $\mathrm{smod}_R$-map $\underline{\mathcal{C}}(a,a) \stackrel{\underline{F}}{\longrightarrow} \underline{\mathcal{D}}(F(a),F(a))$ restricts to a filtered superalgebra morphism
    \begin{equation*}
        \underline{\mathrm{D}}(a) \to \underline{\mathrm{D}}(F(a)).
    \end{equation*}
\end{prop}
\begin{proof}
    Arguing as in the proof of Proposition \ref{prop:diffops-to-diffops}, it suffices to show that $\underline{F}$ fits into a commutative diagram
    \begin{equation*}
        \begin{tikzcd}
        	{\underline{\mathcal{C}}(a,a)} & {\underline{\mathcal{C}}(a^{\otimes 2},a)} & {\underline{\mathcal{C}}(a^{\otimes 3},a)} & \cdots \\
        	{\underline{\mathcal{D}}(F(a),F(a))} & {\underline{\mathcal{D}}(F(a)^{\otimes 2},F(a))} & {\underline{\mathcal{D}}(F(a)^{\otimes 3},F(a))} & \cdots
        	\arrow[from=1-1, to=1-2]
        	\arrow["\underline{F}"', from=1-1, to=2-1]
        	\arrow[from=1-2, to=1-3]
        	\arrow[from=1-2, to=2-2]
        	\arrow[from=1-3, to=1-4]
        	\arrow[from=1-3, to=2-3]
        	\arrow[from=2-1, to=2-2]
        	\arrow[from=2-2, to=2-3]
        	\arrow[from=2-3, to=2-4]
        \end{tikzcd},
    \end{equation*}
    where the horizontal maps are those as in \eqref{eqn:application-of-t-graded}. For this, it suffices to show that the following diagram commutes for all $m \geq 1$:
    \begin{equation}\label{eqn:comm-diag-2}
        \begin{tikzcd}
            {\underline{\mathcal{C}}(a^{\otimes m},a)} & {\underline{\mathcal{D}}(F(a^{\otimes m}),F(a))} & {\underline{\mathcal{D}}(F(a)^{\otimes m},F(a))} \\
            {\underline{\mathcal{C}}(a^{\otimes (m+1)},a)} & {\underline{\mathcal{D}}(F(a^{\otimes (m+1)}),F(a))} & {\underline{\mathcal{D}}(F(a)^{\otimes (m+1)},F(a))}
            \arrow["\underline{F}", from=1-1, to=1-2]
            \arrow[from=1-1, to=2-1]
            \arrow[from=1-2, to=1-3]
            \arrow[from=1-3, to=2-3]
            \arrow["\underline{F}", from=2-1, to=2-2]
            \arrow[from=2-2, to=2-3]
        \end{tikzcd},
    \end{equation}
    where the rightmost horizontal maps are precomposition (in $\underline{\mathcal{D}}$) by the canonical maps $F(a)^{\otimes m} \to F(a^{\otimes m})$ coming from the coherence maps of $F$. Commutativity of the even part of \eqref{eqn:comm-diag-2} was already established in the proof of Proposition \ref{prop:diffops-to-diffops}. The odd part of \eqref{eqn:comm-diag-2} is given by
    \begin{equation}\label{eqn:comm-diag-3}
        \begin{tikzcd}
        	{\mathcal{C}(\pi \otimes a^{\otimes m},a)} & {\mathcal{D}(F(\pi \otimes a^{\otimes m}),F(a))} & {\mathcal{D}(\tau \otimes F(a)^{\otimes m},F(a))} \\
        	{\mathcal{C}(\pi \otimes a^{\otimes (m+1)},a)} & {\mathcal{D}(F(\pi \otimes a^{\otimes (m+1)}),F(a))} & {\mathcal{D}(\tau \otimes F(a)^{\otimes (m+1)},F(a))}
        	\arrow["F", from=1-1, to=1-2]
        	\arrow[from=1-1, to=2-1]
        	\arrow[from=1-2, to=1-3]
        	\arrow[from=1-3, to=2-3]
        	\arrow["F", from=2-1, to=2-2]
        	\arrow[from=2-2, to=2-3]
        \end{tikzcd},
    \end{equation}
    where the rightmost horizontal maps are precomposition (in $\mathcal{D}$) by the maps
    \begin{equation*}
        \tau \otimes F(a)^{\otimes m} \stackrel{\omega \otimes\mathrm{id}}{\longrightarrow} F(\pi) \otimes F(a)^{\otimes m} \to F(\pi \otimes a^{\otimes m}).
    \end{equation*}
    Unravelling the definitions and using properties of the coherence maps of $F$, one can verify that \eqref{eqn:comm-diag-3} commutes.
\end{proof}

We now generalise the classical notion of a BV algebra \cite{penkava1992some,lian1993new,getzler1994batalin}.
\begin{defn}
    Let $\mathcal{C} \in \mathrm{SymMon\Pi\text{-}Cat}_R$. A \textbf{Batalin-Vilkovisky (BV) algebra} in $\mathcal{C}$ is a commutative monoid $a \in \mathcal{C}$ together with an \textit{odd} $\Delta \in \underline{\mathrm{D}}_2(a)$ (called the \textbf{BV operator}) such that $\Delta^2 = 0$ and $\Delta \circ u = 0$, where $u$ is the unit of $a$. The \textbf{associated bracket} is the odd map $\{-,-\}_{\Delta}:a \otimes a \to a$ given by \begin{equation*}
        \includegraphics[scale=1]{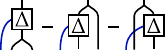}\;.
    \end{equation*}
    Given two BV algebras $(a,\Delta_a)$ and $(b,\Delta_b)$ in $\mathcal{C}$, a \textbf{BV algebra map} $(a,\Delta_a) \to (b,\Delta_b)$ is a map of commutative monoids $f:a \to b$ such that $\Delta_b \circ f = f \circ \Delta_a$ in $\underline{\mathcal{C}}$.
\end{defn}

The following proposition is essentially the idea that lax symmetric monoidal functors send algebras to algebras:

\begin{prop}\label{prop:pi-functor-on-BV-algs}
    Let $F:\mathcal{C} \to \mathcal{D}$ be a lax symmetric monoidal $\Pi$-functor. If $(a,\Delta)$ is a BV algebra in $\mathcal{C}$, then $(F(a),\underline{F}(\Delta))$ is a BV algebra in $\mathcal{D}$. Furthermore, given a map of BV algebras $f:a \to b$, we have that $F(f):F(a) \to F(b)$ is a map of BV algebras.
\end{prop}
\begin{proof}
    This follows immediately from Proposition \ref{prop:diffops-to-diffops-graded}.
\end{proof}

\begin{expl}
    The functor $\underline{\mathcal{C}}(1,-):\mathcal{C} \to \mathrm{smod}_R$ is canonically a lax symmetric monoidal $\Pi$-functor, so $\underline{\mathcal{C}}(1,a)$ is an $R$-BV algebra (in the classical sense) for any BV algebra $a \in \mathcal{C}$.
\end{expl}

\section{A modified necklace Lie algebra and the Poisson story}\label{sec:NLBs}

We first remind the reader of the necklace Lie algebra of a quiver \cite{ginzburg2001,bocklandt2002}. Let $Q = \left[s,t:Q_1 \to Q_0\right]$ be a finite quiver, and denote by $\overline{Q}$ its double. The double of an arrow $a$ in $Q$ is denoted $\overline{a}$. Consider the path algebra $R\overline{Q}$. As an $R$-module, $R\overline{Q}$ is free on the set of paths in $\overline{Q}$ (including the \textit{constant paths} $e_i$ for each vertex $i \in \overline{Q}_0$ --- these are idempotents in $R\overline{Q}$, and $R\overline{Q}$ can be viewed as a $\bigoplus_{i \in \overline{Q}_0}Re_i$-bimodule). Our convention is to read paths from \textit{left to right}: a path in $\overline{Q}$ is a word $a_1\cdots a_n$ in arrows of $\overline{Q}$, where $t(a_i) = s(a_{i+1})$. The multiplication on $R\overline{Q}$ is the usual concatenation of paths.\footnote{Because of our path convention, our path algebra is the opposite of the more-standard notion of path algebra.}

Let $\mathcal{N} := R\overline{Q} / [R\overline{Q} , R\overline{Q}]$ be the quotient of $R\overline{Q}$ by its commutator-submodule, which imposes relations
\begin{equation*}
    a_1\cdots a_n = a_2 \cdots a_n a_1 = \dots = a_n a_1 \cdots a_{n-1}. 
\end{equation*}
Equivalently, $\mathcal{N}$ is the free $R$-module on the set of \textit{cyclic paths} in $\overline{Q}$, where by a `cyclic path' we mean a closed path considered up to cyclic permutation (i.e.\ where its startpoint is forgotten).

The reason for considering the \textit{double} quiver is that it has a canonical antisymmetric pairing $\langle -,- \rangle : \overline{Q}_1 \times \overline{Q}_1 \to \mathbb{Z}$ given by
\begin{equation*}
    \langle a,b \rangle := 
    \begin{cases}
        1, \;\text{if $a \in Q_1$ and $b = \overline{a}$}\\
        -1, \;\text{if $b \in Q_1$ and $a = \overline{b}$}\\
        0, \;\text{otherwise}.
    \end{cases}
\end{equation*}
Using this pairing, the module $\mathcal{N}$ can be given a Lie algebra structure with the following Lie bracket:\footnote{Given that constant paths $e_i$ have length $0$, it is understood that $\mathrm{br}(e_i \otimes -) = 0 = \mathrm{br}(- \otimes e_i)$. The same will apply for all (co)brackets defined later.}
\begin{equation*}
    \mathrm{br}(a_1\cdots a_m \otimes b_1 \cdots b_n) = \sum_{i=1}^m \sum_{j=1}^n \langle a_i, b_j \rangle e_{t(a_i)} a_{i+1} \cdots a_{i-1} b_{j+1} \cdots b_{j-1}.
\end{equation*}
$(\mathcal{N},\mathrm{br})$ is referred to as the \textit{necklace Lie algebra} of the quiver. In fact, the Lie bracket $\mathrm{br}$ can be lifted to the level of the path algebra $R\overline{Q}$, where it takes the form of a Poisson double bracket $\{\!\{-,-\}\!\} : R\overline{Q} \otimes R\overline{Q} \to R\overline{Q} \otimes R\overline{Q}$ in the sense of van den Bergh \cite{van2008double}, defined by
\begin{equation*}
    \{\!\{ a, b \}\!\} = \langle a,b \rangle e_{t(a)} \otimes e_{t(b)}. 
\end{equation*}
Given that $(R\overline{Q},\{\!\{-,-\}\!\})$ is a double Poisson algebra, the usual machinery of van den Bergh applies, recovering the Poisson structures on representation spaces that are related to the necklace Lie algebra $(\mathcal{N},\mathrm{br})$ via the trace maps of Bocklandt, Le Bruyn and Ginzburg.

\subsection{The construction}

Our contributions in this subsection are two-fold. Firstly, we define a modified double bracket on the path algebra $R\overline{Q}$, itself inducing a modified version of the necklace Lie bracket. Secondly, we take a more categorical perspective on the representation side of the story which shows how the regular and modified necklace structures are dual to each other via categorical dualisability.

Consider the double bracket $\{\!\{-,-\}\!\}^+ : R\overline{Q} \otimes R\overline{Q} \to R\overline{Q} \otimes R\overline{Q}$ defined by
\begin{equation}\label{eqn:double-bracket}
    \{\!\{ a, b \}\!\}^+ = \langle a,b \rangle ba \otimes ab. 
\end{equation}
\begin{prop}
    The double bracket $\{\!\{-,-\}\!\}^+$ is a \textit{Poisson} double bracket.
\end{prop}
\begin{proof}
    One needs to show that the associated triple bracket $\{\!\{-,-,-\}\!\}^+$ vanishes. For this, it suffices to show that $\{\!\{a,b,c\}\!\}^+$ vanishes for all arrows $a,b,c$. This is a simple calculation.
\end{proof}
Thus, the Poisson double bracket $\{\!\{-,-\}\!\}^+$ induces a Lie bracket on $\mathcal{N}$ which we denote by $\mathrm{br}^+$. Explicitly:
\begin{equation*}
    \mathrm{br}^+(a_1\cdots a_m \otimes b_1 \cdots b_n) = \sum_{i=1}^m \sum_{j=1}^n \langle a_i, b_j \rangle a_i a_{i+1} \cdots a_{i-1} a_i b_j b_{j+1} \cdots b_{j-1} b_j.
\end{equation*}
We call this bracket the \textit{augmented necklace bracket}, and call $(\mathcal{N},\mathrm{br}^+)$ the \textit{augmented necklace Lie algebra}. For homogeneity, we can denote by $\mathrm{br}^-$ the usual necklace bracket $\mathrm{br}$.
\begin{rem}
    Note that the usual necklace bracket can be seen to `resolve' intersections of cyclic paths by removing pairs of arrows in involution, while the augmented bracket `worsens' them by instead inserting such pairs.
\end{rem}

We now turn to the representation side of the story. In what follows, let $\mathcal{C}$ be an $R$-linear symmetric monoidal category. For simplicity we assume that $\mathcal{C}$ is $\otimes$-cocomplete.

\begin{rem}
    Our assumptions on $\mathcal{C}$ can be weakened: we don't really need $\mathcal{C}$ to possess \textit{all} colimits, nor $\otimes$ to preserve \textit{all} of them; only enough to be additive and so that free commutative monoids can be modelled via Construction \ref{constr:free-com-mon}.
\end{rem}

For each vertex $i$ of the quiver $\overline{Q}$, fix $c_i \in \mathcal{C}$ a dualisable object. Recall that \textit{dualisable} means that there exists an object $c_i^\vee \in \mathcal{C}$ together with morphisms $\ev:c_i^\vee \otimes c_i \to 1$ and $\coev:1 \to c_i \otimes c_i^\vee$ (called \textit{evaluation} and \textit{coevaluation}) satisfying the usual \textit{snake identities}. This data is unique in an appropriate sense, and hence we are safe to refer to $c_i^\vee$ as \textit{the} dual of $c_i$.

For each arrow $a$ in $\overline{Q}$, define $h_a := c_{t(a)} \otimes c_{s(a)}^\vee$. This can been seen as the internal hom object $\left[c_{s(a)},c_{t(a)}\right]$, since duality data exhibits an adjunction $-\otimes c \dashv - \otimes c^\vee$. In particular, the $1$-points of $h_a$ are precisely the morphisms $c_{s(a)} \to c_{t(a)}$ in $\mathcal{C}$. Following this observation, defining $h$ to be the biproduct\footnote{This exists since the category $\mathcal{C}$ is additive and the quiver $\overline{Q}$ is finite.} $\bigoplus_{a \in \overline{Q}_1} h_a$, we see that the $1$-points of $h$ are the same as representations of $\overline{Q}$ in $\mathcal{C}$ with vertices painted by the objects $c_i$. For this reason, we refer to $h$ as a \textit{representation variety} of the double quiver.

Our next step is to construct an object in $\mathcal{C}$ that models `functions on the representation variety'. Note that $h$ is itself dualisable with dual given by
\begin{equation}
    h^\vee = \bigoplus_{a \in \overline{Q}_1} \underbrace{c_{s(a)} \otimes c_{t(a)}^\vee}_{h_a^\vee}.
\end{equation}
As such, mimicking polynomial functions on a vector space, the object we consider is the free commutative monoid on $h^\vee$ internal to $\mathcal{C}$, $Sh^\vee$. This commutative monoid can be equipped with a Poisson bracket which we relate to the necklace bracket.

A bi-derivation $Sh^\vee \otimes Sh^\vee \to Sh^\vee$ is equivalently a map $h^\vee \otimes h^\vee \to Sh^\vee$ (see \S\ref{sec:bi-differential operators}). The data of such a map is then the same as the data of maps
\begin{equation}
    h_a^\vee \otimes h_b^\vee \to Sh^\vee
\end{equation}
for all arrows $a$ and $b$.

\begin{itemize}
    \item Suppose that $a$ and $b$ are two arrows in involution (i.e. $a=\overline{b}$ or $b = \overline{a}$). Then they in particular have opposite source and target vertices, and we have
    \begin{equation}
        h_a^\vee = c_i \otimes c_j^\vee \quad \text{and} \quad h_b^\vee = c_j \otimes c_i^\vee,
    \end{equation}
    for appropriate vertices $i,j$. Evaluation provides a map $h_a^\vee \otimes h_b^\vee \to 1$. In diagrammatic notation this map is simply
    \begin{equation*}
        \vcenter{\hbox{\includegraphics[scale=1]{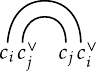}}}\;,
    \end{equation*}
    which for brevity we draw as 
    \begin{equation}\label{eqn:trace-pairing}
        \vcenter{\hbox{\includegraphics[scale=1]{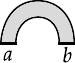}}}\;,
    \end{equation}
    where the solid gray colour suggests that we should view such pairs of strands as one object (e.g.\ $h_a^\vee$).

    \item When $a$ and $b$ are not in involution, the only map $h_a^\vee \otimes h_b^\vee \to 1$ that we can in specify in general is $0$.
\end{itemize}
\begin{rem}
    Equation \eqref{eqn:trace-pairing} is simply the evaluation map for duality data $(h_b^\vee)^\vee = h_a^\vee$. In the classical case $\mathcal{C} = \mathrm{vect}_k$, this map is called the \textit{trace pairing}: the spaces $h_a^\vee$ and $h_b^\vee$ are identified with hom spaces, and \eqref{eqn:trace-pairing} is simply `compose and trace'.
\end{rem}
\begin{defn}
    Define $\{-,-\}:Sh^\vee \otimes Sh^\vee \to Sh^\vee$ as the bi-derivation obtained from the map $h^\vee \otimes h^\vee \to Sh^\vee$ whose $ab$'th component $h_a^\vee \otimes h_b^\vee \to Sh^\vee$ is
    \begin{equation}\label{eqn:P-bracket-usual}
        \vcenter{\hbox{\includegraphics[scale=1]{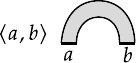}}}\;,
    \end{equation}
    where we have implicitly composed with $1 \to Sh^\vee$ at the top.
\end{defn}
Notice that $\{-,-\}$ as defined is antisymmetric, due to antisymmetry of $\langle -,- \rangle$. Thus, by the discussion of \S\ref{sec:bi-differential operators}, to show that $\{-,-\}$ is a \textit{Poisson} bracket it suffices to show that $\{-,\{-,-\}\} \circ \mathfrak{S}_3$ vanishes on $(h^\vee)^{\otimes 3}$. But $\{-,-\}$ is a degree $-2$ map with respect to polynomial grading, and so one sees that it satisfies the Jacobi identity purely due to degree reasons. Therefore, we have the following:
\begin{prop}
    $(Sh^\vee,\{-,-\})$ is a Poisson algebra in $\mathcal{C}$.
\end{prop}

This Poisson algebra is related to the necklace Lie algebra $(\mathcal{N},\mathrm{br})$ in a nice way. For homogeneity, we work with the induced Poisson algebra $S\mathcal{N}$. Now, $Sh^\vee$ lives in $\mathcal{C}$ while $S\mathcal{N}$ lives in $\mathrm{mod}_R$, so to be able to compare these two structures we apply the linear lax symmetric monoidal functor $\mathcal{C}(1,-):\mathcal{C} \to \mathrm{mod}_R$, which gives an $R$-Poisson algebra $\mathcal{C}(1,Sh^\vee)$ whose bracket we denote by $\{-,-\}$ also.

Next, let $\mathrm{tr}:S\mathcal{N} \to \mathcal{C}(1,Sh^\vee)$ be the map of commutative $R$-algebras defined on the generating set of cyclic paths as
\begin{equation*}
    \vcenter{\hbox{\includegraphics[scale=1]{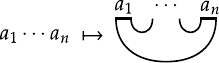}}} \quad \text{and} \quad
    \vcenter{\hbox{\includegraphics[scale=1]{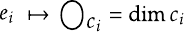}}}\;,
\end{equation*}
where we implicitly composed with $h_{a_1}^\vee \otimes \cdots \otimes h_{a_n}^\vee \to (h^\vee)^{\otimes n} \to Sh^\vee$ or $1 \to Sh^\vee$ in the diagrams, as appropriate. This map behaves well with respect to the necklace bracket:
\begin{prop}\label{prop:interwining-of-brackets}
    $\mathrm{tr}:S\mathcal{N} \to \mathcal{C}(1,Sh^\vee)$ intertwines the Poisson brackets $\mathrm{br}$ and $\{-,-\}$.
\end{prop}
\begin{proof}
    It suffices to show that $\mathcal{N} \to \mathcal{C}(1,Sh^\vee)$ intertwines the Lie bracket $\mathrm{br}$ on $\mathcal{N}$ with the Poisson bracket $\{-,-\}$ on $\mathcal{C}(1,Sh^\vee)$. Using that $\{-,-\}$ is a bi-derivation, we find that
    \begin{align*}
        \{\mathrm{tr}(a_1 \cdots a_m) \otimes \mathrm{tr}(b_1 \cdots b_n)\}
        &= \;
        \vcenter{\hbox{\includegraphics[scale=1]{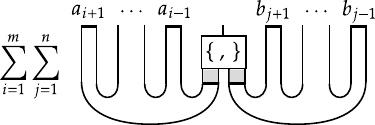}}}\\
        =\;
        \vcenter{\hbox{\includegraphics[scale=1]{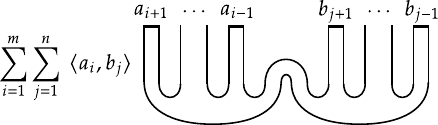}}}
        &=\;
        \vcenter{\hbox{\includegraphics[scale=1]{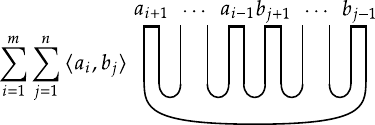}}}\\
        &= \mathrm{tr}(\mathrm{br}(a_1\cdots a_m \otimes b_1 \cdots b_n)),
    \end{align*}
    where we implicitly mapped to $Sh^\vee$ in all the diagrams  --- as usual.
\end{proof}

In the proof of Proposition \ref{prop:interwining-of-brackets} we did not spell out what happens when there are occurences of constant paths --- let us assure the reader that everything works out fine. We will continue to suppress these edge cases, such as in the proof of Theorem \ref{thm:intertwining-BV-ops}.

\begin{expl}
    Suppose $R = k$ is a field and take $\mathcal{C} = \mathrm{mod}_R = \mathrm{vect}_k$. Then each $c_i$ is necessarily a finite dimensional vector space and $\mathcal{C}(1,Sh^\vee) \cong \mathscr{O}[\mathrm{Rep}_{\overline{Q}}(c)]$ is the algebra of polynomial functions on the space of $k$-linear representations of $\overline{Q}$ with dimension vector $(\dim c_i)_{i}$. The map $\mathrm{tr}$ is simply the trace map of Bocklandt, Le Bruyn and Ginzburg \cite{ginzburg2001,bocklandt2002}: indeed, $\mathrm{tr}(a_1 \cdots a_m)$ is the polynomial function that sends a representation $f$ to $\mathrm{tr}(f_{a_m} \circ \cdots \circ f_{a_1})$.
\end{expl}

This gives the well-known manifestation of the necklace Lie bracket $\mathrm{br}$ in the representation point of view, albeit through a different perspective via more general categorical language. But what about the augmented necklace bracket?

\begin{defn}
    Define $\{-,-\}^+:Sh^\vee \otimes Sh^\vee \to Sh^\vee$ as the (quartic) bi-derivation obtained from the map $h^\vee \otimes h^\vee \to Sh^\vee$ whose $ab$'th component $h_a^\vee \otimes h_b^\vee \to Sh^\vee$ is
    \begin{equation}\label{eqn:P-bracket-augmented}
        \vcenter{\hbox{\includegraphics[scale=1]{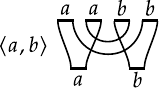}}}\;,
    \end{equation}
    where we have implicitly composed with $h_a^\vee \otimes h_a^\vee \otimes h_b^\vee \otimes h_b^\vee \to Sh^\vee$ at the top.
\end{defn}

\begin{prop}\label{prop:quartic-poisson-bracket}
    $(Sh^\vee,\{-,-\}^+)$ is a Poisson algebra in $\mathcal{C}$.
\end{prop}
\begin{proof}
    Antisymmetry of $\{-,-\}^+$ is immediate, and so to show that it is a Poisson bracket it suffices to show that $\{-,\{-,-\}^+\}^+ \circ \mathfrak{S}_3$ vanishes on $h_a^\vee \otimes h_b^\vee \otimes h_c^\vee$ for all arrows $a,b,c$. Calculating, we have
    \begin{align*}
        \includegraphics[scale=1]{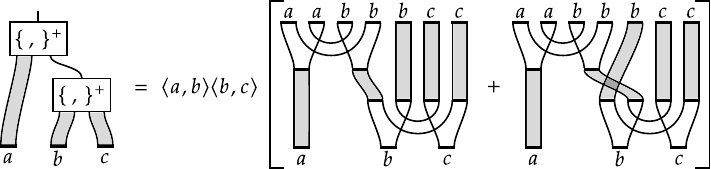}\\
        \vcenter{\hbox{\includegraphics[scale=1]{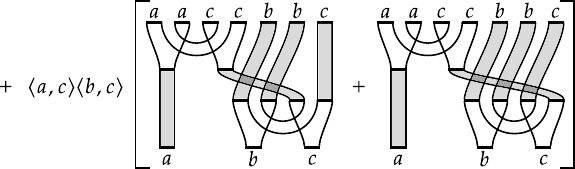}}}\;.
    \end{align*}
    Simplifying the string diagrams and permuting the (pairs) of strands at the tops (which we can do as we implicitly map into $Sh^\vee$), this becomes
    \begin{align*}
        \includegraphics[scale=1]{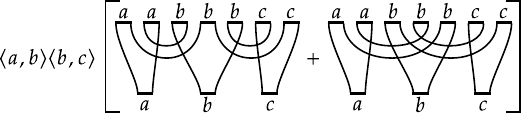}\\
        \vcenter{\hbox{\includegraphics[scale=1]{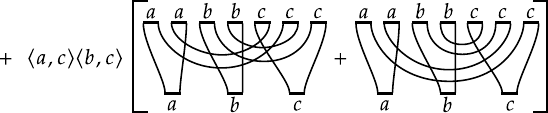}}}\;.
    \end{align*}
    Similarly, we get
    \begin{align*}
        \includegraphics[scale=1]{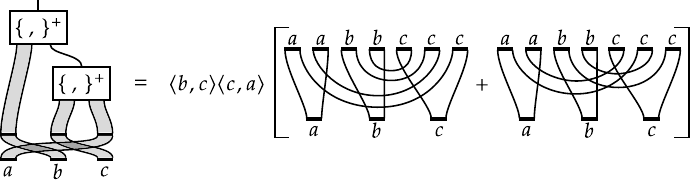}\\
        \vcenter{\hbox{\includegraphics[scale=1]{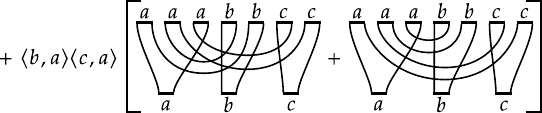}}}\;
    \end{align*}
    and
    \begin{align*}
        \includegraphics[scale=1]{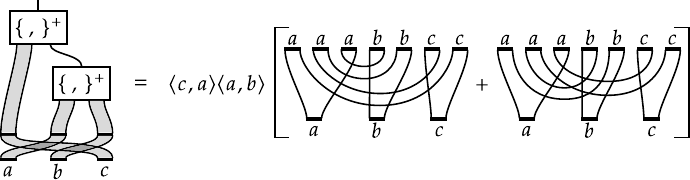}\\
        \vcenter{\hbox{\includegraphics[scale=1]{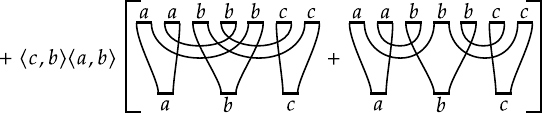}}}\;.
    \end{align*}
    Thus, the result follows by antisymmetry of $\langle-,-\rangle$.
\end{proof}

As before, applying the functor $\mathcal{C}(1,-)$ to the Poisson algebra $(Sh^\vee,\{-,-\}^+)$ yields an $R$-Poisson algebra. The underlying commutative algebra is again $\mathcal{C}(1,Sh^\vee)$ and, abusing notation, we denote its Poisson bracket by $\{-,-\}^+$ too.

\begin{prop}
    The map $\mathrm{tr}:S\mathcal{N} \to \mathcal{C}(1,Sh^\vee)$ intertwines the Poisson brackets $\mathrm{br}^+$ and $\{-,-\}^+$.
\end{prop}
\begin{proof}
    \begin{align*}
        \{\mathrm{tr}(a_1 \cdots a_m) \otimes \mathrm{tr}(b_1 \cdots b_n)\}^+
        &= \;
        \vcenter{\hbox{\includegraphics[scale=1]{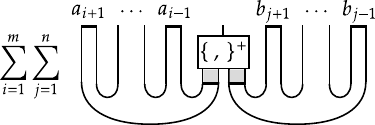}}}\\
        &=\;
        \vcenter{\hbox{\includegraphics[scale=1]{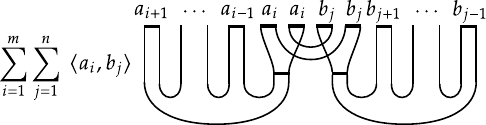}}}\\
        &=\;
        \vcenter{\hbox{\includegraphics[scale=1]{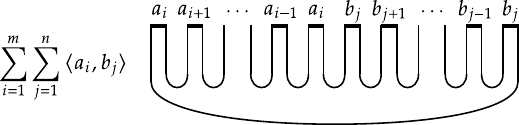}}}\\
        &= \mathrm{tr}(\mathrm{br}^+(a_1\cdots a_m \otimes b_1 \cdots b_n)).
    \end{align*}
\end{proof}

To summarise, we have two natural Lie brackets on $\mathcal{N}$, $\mathrm{br}^-$ and $\mathrm{br}^+$ (and hence two natural Poisson brackets on $S\mathcal{N}$). After painting the vertices of $\overline{Q}$ with dualisable objects of $\mathcal{C}$, we obtain a commutative monoid $Sh^\vee$ which itself carries two natural Poisson brackets: $\{-,-\}^- := \{-,-\}$ \eqref{eqn:P-bracket-usual} and $\{-,-\}^+$ \eqref{eqn:P-bracket-augmented}. The map $\mathrm{tr}:S\mathcal{N} \to \mathcal{C}(1,Sh^\vee)$ intertwines the brackets $\mathrm{br}^\pm$ and $\{-,-\}^\pm$. It is in this sense that the usual and augmented necklace brackets are two sides of the same coin: $\{-,-\}^-$ is canonically defined with evaluation maps, while $\{-,-\}^+$ is canonically defined with coevaluation maps. In fact, \textit{the categorical construction completely subsumes the necklace story}, as the following proposition shows.
\begin{prop}\label{prop:injectivity-of-tr}
    There exists a cocomplete closed linear symmetric monoidal category $\mathcal{C}$ with a choice of dualisable objects $c_i \in \mathcal{C}$ for which the trace map $\mathrm{tr}:S\mathcal{N} \to \mathcal{C}(1,Sh^\vee)$ is \textit{injective}.
\end{prop}
\begin{proof}
    The trick is to take $\mathcal{C}$ to be 'free enough', so that string diagrams, which have previously only been a \textit{notation} for morphisms, are the \textit{actual morphisms}. To this end, let $\mathrm{Diag}_{\overline{Q}}$ be the (small) category whose objects are finite tuples of vertices of $\overline{Q}$ --- which we think of as collections of $\overline{Q}_0$-coloured points --- and whose maps are (homotopy classes of) unoriented $\overline{Q}_0$-coloured `string diagrams' between such tuples. For example, the following is a map $(i,j,k) \to (l,i,k,l,j)$:
        \begin{equation*}
            \vcenter{\hbox{\includegraphics[scale=1]{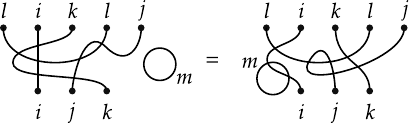}}}\;.
        \end{equation*}
    These diagrams are determined by their coloured loops and by which same-colour vertices are connected. Composition is given by vertical stacking. Notice that horizontal juxtaposition renders $\mathrm{Diag}_{\overline{Q}}$ a symmetric monoidal category with unit the empty tuple $\emptyset$, and that each vertex of $\overline{Q}$ --- viewed as an single-element tuple --- is symmetrically self-dual.
    Denote by $\mathcal{A}$ the $R$-linear symmetric monoidal category obtained from   $\mathrm{Diag}_{\overline{Q}}$ by taking free $R$-modules of all hom sets and consider its ($\mathrm{mod}_R$-enriched) presheaf category $\widehat{\mathcal{A}}:=[\mathcal{A}^{\op},\mathrm{mod}_R]$, its cocompletion\footnote{This is a well-known result in enriched category theory, see for instance \cite[Thm.\ 4.51]{kelly1982basic}.} (in the stronger, enriched sense). Day convolution gives $\widehat{\mathcal{A}}$ a closed symmetric monoidal structure, relative to which the Yoneda embedding $\yo:\mathcal{A} \hookrightarrow \widehat{\mathcal{A}}$ is strong symmetric monoidal \cite{day2006closed}.

    Take $\mathcal{C} = \widehat{\mathcal{A}}$, and let $c_i = \yo(i)$ for each vertex $i$, which is symmetrically self-dual. Then for each arrow $a$ we have
    \begin{equation*}
        h_a^\vee = c_{s(a)} \otimes c_{t(a)}^\vee = \yo(s(a)) \otimes \yo(t(a))^\vee \cong \yo(s(a),t(a)).
    \end{equation*}
    When the arrow $a$ is remembered (which will be the case when working with the $\overline{Q}_1$-indexed direct sum $h^\vee$) we abbreviate the pair $(s(a),t(a))$ by `$a$' itself and depict it as
    \begin{equation*}
        \includegraphics[scale=1]{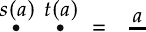},
    \end{equation*}
    akin to our earlier diagrammatic notation. Consider now the $R$-Poisson algebra $\mathcal{C}(1,Sh^\vee)$. We have
    \begin{equation*}
        \mathcal{C}(1,Sh^\vee) = \widehat{\mathcal{A}}\left(\yo(\emptyset),Sh^\vee\right)
        \cong (Sh^\vee)(\emptyset)
        \cong \bigsqcup_n\left(\bigoplus_{a_1,\dots,a_n}\left(h_{a_1}^\vee \otimes_{\widehat{\mathcal{A}}} \cdots \otimes_{\widehat{\mathcal{A}}} h_{a_n}^\vee\right) \right)_{S_n}(\emptyset),
    \end{equation*}
    where we used the Yoneda lemma, and where the colimits are computed in $\widehat{\mathcal{A}}$. But colimits in $\widehat{\mathcal{A}} = [\mathcal{A}^\op,\mathrm{mod}_R]$ are computed object-wise, so this is nothing but
    \begin{equation*}
        \bigsqcup_n\left(\bigoplus_{a_1,\dots,a_n}\left(h_{a_1}^\vee \otimes_{\widehat{\mathcal{A}}} \cdots \otimes_{\widehat{\mathcal{A}}} h_{a_n}^\vee\right)(\emptyset) \right)_{S_n},
    \end{equation*}
    where the colimits are now computed in $\mathrm{mod}_R$. Notice that
    \begin{equation*}
        \left(h_{a_1}^\vee \otimes_{\widehat{\mathcal{A}}} \cdots \otimes_{\widehat{\mathcal{A}}} h_{a_n}^\vee\right)(\emptyset)
        \cong
        \yo(a_1,\dots,a_n)(\emptyset) 
        = \mathcal{A}(\emptyset,(a_1,\dots,a_n))
    \end{equation*}
    is the free $R$-module on the set of string diagrams $\emptyset \to (a_1,\dots,a_n)$, and that the corresponding $S_n$-action on $\bigoplus_{a_1,\dots,a_n}\mathcal{A}(\emptyset,(a_1,\dots,a_n)) \cong \left\langle \bigsqcup_{a_1,\dots,a_n} \mathrm{Diag}_{\overline{Q}}(\emptyset,(a_1,\dots,a_n)) \right\rangle_R$ is simply the one induced by post-composition using the symmetry of $\mathrm{Diag}_{\overline{Q}}$, for example,
    \begin{equation}\label{eqn:012}
        \vcenter{\hbox{\includegraphics[scale=1]{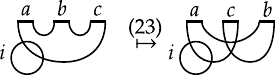}}}\;.
    \end{equation}
    Thus we have
    \begin{equation}\label{eqn:013}
        \mathcal{C}(1,Sh^\vee) \cong
        \left\langle
            \bigsqcup_n \left(\bigsqcup_{a_1,\dots,a_n} \mathrm{Diag}_{\overline{Q}}\left(\emptyset,(a_1,\dots,a_n)\right)\right)\Big/ S_n 
        \right\rangle_R,
    \end{equation}
    the free $R$-module on the set of $\overline{Q}_0$-coloured string diagrams from $\emptyset$ to finite tuples of arrows $(a_1,\dots,a_n)$\footnote{More precisely, to the tuple of vertices $(s(a_1),t(a_1),\dots,s(a_n),t(a_n))$ coloured pairwise by the arrows $a_i$.} considered up to the permutation action showcased in \eqref{eqn:012}. 
    
    We now turn to the trace map. The module $S\mathcal{N}$ is free on the set of (symmetric) products of cyclic paths in $\overline{Q}$. Consider two such basis elements 
    \begin{equation*}
        x = (e_{i_1})\cdots (e_{i_k}) \cdot (a_{1}^{1}\cdots a_{n_1}^{1}) \cdots (a_{1}^{p}\cdots a_{n_p}^{p}) \quad \text{and} \quad
        y = (e_{j_1})\cdots (e_{j_l}) \cdot (b_{1}^{1}\cdots b_{m_1}^{1}) \cdots (b_{1}^{q}\cdots b_{m_q}^{q}),
    \end{equation*}
    with each $n_\alpha,m_\beta > 0$.
    Under \eqref{eqn:013}, $\mathrm{tr}(x)$ and $\mathrm{tr}(y)$ are basis elements determined by the coloured diagrams
    \begin{equation*}
        \vcenter{\hbox{\includegraphics[scale=1]{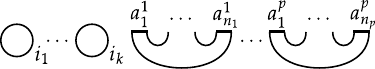}}}\quad \text{and} \quad
        \vcenter{\hbox{\includegraphics[scale=1]{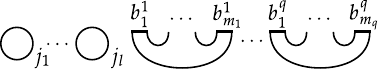}}}\;.
    \end{equation*}
    Suppose that $\mathrm{tr}(x) = \mathrm{tr}(y)$. Then these diagrams are related via the action of a permutation. In particular, $n_1 + \dots + n_p = m_1 + \dots + m_q$. Considering the strands and involved arrows shows that $p=q$ and that there is a permutation $\tau$ such that $a_1^\alpha \cdots a_{n_{\alpha}}^\alpha = b_1^{\tau(\alpha)}\cdots b_{m_{\tau(\alpha)}}^{\tau(\alpha)}$ as \textit{cyclic} paths for each $\alpha=1,\dots,p$. Meanwhile, considering the $\overline{Q}_0$-coloured loops shows that $k = l$ and that there is a permutation $\omega$ such that $i_\alpha = j_{\omega(\alpha)}$ for each $\alpha=1,\dots,k$. As a result, we have $x = y$.

    We therefore see that $\mathrm{tr}$ restricts to an \textit{injective} map between $R$-bases, and hence that $\mathrm{tr}$ is itself injective.
\end{proof}

\subsection{Invariant functions}
Let $G = \prod_{i \in \overline{Q}_0} \mathrm{Aut}(c_i)$. The projections $G \to \mathrm{Aut}(c_i)$ exhibit an action of $G$ on each object $c_i$. $G$ also acts on their duals $c_i^\vee$ by the maps $G \to \mathrm{Aut}(c_i^\vee)$ sending $(\phi_j)_{j \in \overline{Q}_0} \mapsto (\phi_i^{-1})^\vee$. As such, $G$ acts on each internal hom object $h_a$ and hence on $h$. All considered, $Sh^\vee$ carries a natural $G$-action. Denote by $(Sh^\vee)^G \hookrightarrow Sh^\vee$ the equaliser of this action.\footnote{Here we used that the actions of $G$ and $S_n$ on $(h^\vee)^{\otimes n}$ commute.} Since $Sh^\vee$ models `functions on $h$', one should think of $(Sh^\vee)^G$ as modelling the `$G$-invariant functions'.
\begin{prop}
    The monic map $(Sh^\vee)^G \hookrightarrow Sh^\vee$ exhibits $(Sh^\vee)^G$ as a sub-Poisson algebra of $(Sh^\vee,\{-,-\}^\pm)$. Furthermore, the map $\mathrm{tr}:S\mathcal{N} \to \mathcal{C}(1,Sh^\vee)$ factors through $\mathcal{C}(1,Sh^\vee)^G \cong \mathcal{C}(1,(Sh^\vee)^G) \hookrightarrow \mathcal{C}(1,Sh^\vee)$.
\end{prop}
\begin{proof}
    It is clear that the unit $1 \to Sh^\vee$ equalises the $G$-action. Notice that the composite
    \begin{equation*}
        \begin{tikzcd}
        	{(Sh^\vee)^G \otimes (Sh^\vee)^G} & {Sh^\vee \otimes Sh^\vee} & {Sh^\vee}
        	\arrow[from=1-1, to=1-2]
        	\arrow[from=1-2, to=1-3]
        \end{tikzcd}
    \end{equation*}
    equalises the $G$-action when the rightmost map is either multiplication or the Poisson bracket $\{-,-\}^\pm$. This proves the first claim. For the second, given $\phi \in G$, one simply calculates that $\phi \cdot \mathrm{tr}(a_1 \cdots a_n) = \mathrm{tr}(a_1 \cdots a_n)$ and $\phi \cdot \mathrm{tr}(e_i) = \mathrm{tr}(e_i)$.
\end{proof}
It is natural to ask for when $S\mathcal{N} \to \mathcal{C}(1,Sh^\vee)^G$ is surjective. When $\mathcal{C}=\mathrm{vect}_k$, this question would be a problem in invariant theory.

\section{IBL algebra structure and the BV story}\label{subsec:lie-bialg-structure}
In \cite{schedler2005hopf} Schedler showed that the necklace Lie algebra $(\mathcal{N},\mathrm{br})$ carries a Lie cobracket $\delta = \delta^-$ such that the triple $(\mathcal{N},\mathrm{br},\delta)$ is an \textit{involutive Lie bialgebra} (IBL algebra). We assume from now on that $2 \in R$ is invertible, and in particular this allows us to match the normalisation conventions of the cobracket.\footnote{Invertibility of $2$ also means that we need not worry about whether $2$ is a zero divisor for certain $R$-modules.} The cobracket $\delta:\mathcal{N} \to \mathcal{N} \otimes \mathcal{N}$ is given by the simple formula
\begin{equation*}
    \delta(a_1\cdots a_n) = \frac{1}{2}\sum_{i,j=1}^n \langle a_i, a_j \rangle  e_{t(a_i)}a_{i+1} \cdots a_{j-1} \otimes e_{t(a_{j})} a_{j+1} \cdots a_{i-1},
\end{equation*}
which can be interpreted as resolving self-intersections of cyclic paths. Given that the necklace bracket is related to a simple structure on the representation variety, a natural question was then to ask whether the cobracket corresponds to some meaningful structure on the representation variety too. Recently, in \cite{perry2024graded}, this question was answered in the following sense: if one combines the necklace bracket and cobracket into a single BV operator on $\bigwedge \mathcal{N}$, then this BV operator can be related to a canonical (constant coefficients) BV structure on a certain representation variety via an `odd' trace map.

In this subsection we show that the augmented necklace Lie algebra can be given a natural IBL structure too, and that the resulting BV operator can be related to a canonical (quartic) BV structure in the same way, extending the results from \cite{perry2024graded}. By doing the constructions in more general settings, we in the process generalise the earlier results.

Define the \textit{augmented necklace cobracket} $\delta^+:\mathcal{N} \to \mathcal{N} \otimes \mathcal{N}$ by
\begin{equation*}
    \delta^+(a_1 \cdots a_n) = \frac{1}{2} \sum_{i,j=1}^n \langle a_i, a_j \rangle a_i a_{i+1} \cdots a_{j-1} a_j \otimes a_j a_{j+1} \cdots a_{i-1} a_i.
\end{equation*}
\begin{prop}
    The triple $(\mathcal{N},\mathrm{br}^+,\delta^+)$ is an involutive Lie bialgebra.
\end{prop}
\begin{proof}
    This is a tedious calculation that we do not spell out here. Instead, we refer the reader to Appendix C of \cite{perry2024graded}, which gives a proof of the IBL conditions for the usual necklace Lie bialgebra. To prove the proposition at hand, one mimics that calculation. While there are many extra terms appearing in the augmented case, they cancel pairwise. 
\end{proof}
Thus, the module $\mathcal{N}$ admits two natural IBL structures: one given by $(\mathrm{br}^-,\delta^-)$, and one given by $(\mathrm{br}^+,\delta^+)$. Each IBL structure gives rise to a BV operator on $\bigwedge \mathcal{N}$ in the usual way, which goes as follows:
\begin{itemize}
    \item Under $\mathrm{mod}_R \to \mathrm{smod}_R$, regard $(\mathcal{N},\mathrm{br}^\pm,\delta^\pm)$ as an IBL algebra in $\mathrm{smod}_R$ concentrated in even degree.
    \item Shifting $\mathcal{N} \mapsto \Pi\mathcal{N}$, we can view $\mathrm{br}^\pm$ and $\delta^\pm$ as \textit{odd} maps $S^2 \Pi \mathcal{N} \rightleftarrows \Pi \mathcal{N}$ in $\underline{\mathrm{smod}}_R$.
    \item Extend $\mathrm{br}^\pm:S^2 \Pi \mathcal{N} \to \Pi \mathcal{N}$ and $\delta^\pm: \Pi \mathcal{N} \to S^2 \Pi \mathcal{N}$ to odd 2nd and 1st order differential operators on $S\Pi\mathcal{N} \cong \bigwedge \mathcal{N}$, respectively.\footnote{It is understood that $\mathrm{br}^\pm$ vanishes on $S^{\leq 1} \Pi \mathcal{N}$, and that $\delta^\pm$ vanishes on $S^0 \Pi \mathcal{N}$.}
    \item The IBL algebra axioms imply that the odd 2nd order differential operator $\Delta^\pm := 2 \mathrm{br}^\pm - \delta^\pm$ squares to zero, so that $(S \Pi \mathcal{N},\Delta^\pm)$ is an $R$-BV algebra.\footnote{The reason for the coefficients `$2$' and `$-1$' is only for Theorem \ref{thm:intertwining-BV-ops} later.}
\end{itemize}
Explicitly, the BV operator $\Delta^\pm$ is given by
\begin{equation*}
    \Delta^\pm(x_1 \cdots x_n) = \sum_{i<j}(-1)^{i+j+1}2\mathrm{br}^\pm(x_ix_j)x_1 \cdots \hat{x}_i \cdots \hat{x}_j \cdots x_n - \sum_i(-1)^{i+1}\delta^\pm(x_i)x_1 \cdots \hat{x}_i \cdots x_n ,
\end{equation*}
for all elements $x_1 ,\dots,x_n \in \Pi\mathcal{N}$, where the hats denote omission.

Let us turn to the representation side of the story. Let $\mathcal{C}$ be a symmetric monoidal $\Pi$-category, and for simplicity assume that $\mathcal{C}$ is $\otimes$-cocomplete. For each vertex $i$ of $\overline{Q}$, fix a dualisable object $c_i \in \mathcal{C}$ together with an \textit{odd} map $\iota_i \in \underline{\mathcal{C}}(c_i,c_i)$ such that $\iota_i^2 = \mathrm{id}$ (i.e.\ an \textit{odd} involution). The reason for fixing these involutions is that the shift $\mathcal{N} \mapsto \Pi \mathcal{N}$ will need to be witnessed from the representation side. As before, we can consider the objects $h_a^\vee = c_{s(a)} \otimes c_{t(a)}^\vee$. Denote by $\widetilde{h}_{a}^\vee$ the coequaliser of the two maps $\pi \otimes h_a^\vee \to h_a^\vee$ given by
\begin{equation}\label{eqn:019}
    \vcenter{\hbox{\includegraphics[scale=1]{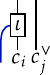}}}\quad \text{and} \quad
    \vcenter{\hbox{\includegraphics[scale=1]{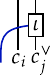}}}\;,
\end{equation}
where here $i = s(a)$ and $j = t(a)$. We have suppressed some notation: we really should write $\iota_i$ and $\iota_j^\vee$ above, where the dual map $\iota_j^\vee:\pi \otimes c_j^\vee \to c_j^\vee$ is defined to be
\begin{equation}\label{eqn:020}
    \vcenter{\hbox{\includegraphics[scale=1]{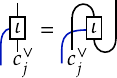}}}\;.
\end{equation}
That is, our convention is to keep $\pi$ fixed when taking the dual map (as opposed to considering a map $c_j^\vee \to c_j^\vee \otimes \pi^\vee$).\footnote{This avoids confusion: while $\pi$ is dualisable, $(\xi,\xi^{-1})$ is not necessarily duality data --- see Remark \ref{rem:dualisable?}.} While this convention behaves better with respect to the definition of odd maps in $\underline{\mathcal{C}}$, it means that one needs to be careful with signs: we have $\iota_i^2 = \mathrm{id}$ in $\underline{\mathcal{C}}$, but
\begin{equation*}
    \vcenter{\hbox{\includegraphics[scale=1]{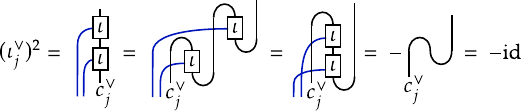}}}\;,
\end{equation*}
where we used $\sigma_{\pi,\pi} = -\mathrm{id}_{\pi \otimes \pi}$.

The $1$-points of the dual of $\widetilde{h}_a^\vee$, $\widetilde{h}_a$, which itself is an equaliser, are precisely the maps $c_i \to c_j$ in $\mathcal{C}$ that intertwine the involutions $\iota$. Put $\widetilde{h} = \bigoplus_{a}\widetilde{h}_a$. The $1$-points of $\widetilde{h}$ are the representations of $\overline{Q}$ (with vertices painted by the objects $c_i$) that intertwine the involutions $\iota$. This $\widetilde{h}$ is the `representation variety' of interest for us. To model `functions on $\widetilde{h}$', consider the free commutative monoid $S\widetilde{h}^\vee$ in $\mathcal{C}$. We will describe two BV operators on $S\widetilde{h}^\vee$.

By Theorem \ref{thm:!existence-of-diffops-graded}, an odd 2nd order differential operator $\Delta \in \underline{\mathrm{D}}_2(S \widetilde{h}^\vee)$ is equivalently the data of a map $\pi \otimes S^{\leq 2} \widetilde{h}^\vee \to S \widetilde{h}^\vee$ in $\mathcal{C}$. Setting this map to be $0$ on $\pi \otimes S^{\leq 1} \widetilde{h}^\vee$, it remains to specify a map $\pi \otimes S^2 \widetilde{h}^\vee \to S\widetilde{h}^\vee$. Since $\otimes$ preserves colimits, this is the same thing as a map $\phi:\pi \otimes h^\vee \otimes h^\vee \to S\widetilde{h}^\vee$ satisfying the following two properties:
\begin{enumerate}
    \item\label{itm:006} $\phi$ coequalises the $S_2$-action on $h^\vee \otimes h^\vee$; and
    \item\label{itm:005} The $ab$'th component of $\phi$ coequalises the two maps $\pi \otimes (\pi \otimes h_a^\vee) \otimes h_b^\vee \to \pi \otimes h_a^\vee \otimes h_b^\vee$ given by
    \begin{equation*}
        \vcenter{\hbox{\includegraphics[scale=1]{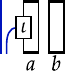}}} \quad \text{and} \quad
        \vcenter{\hbox{\includegraphics[scale=1]{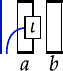}}}\;.
        \; \footnote{These properties together imply that $\phi$ factors through $h^\vee \twoheadrightarrow \widetilde{h}^\vee$ applied to the \textit{each} tensor factor $h^\vee$.}
    \end{equation*}
\end{enumerate}
Helpful for later calculations will be the following notation:
\begin{notation}
    Given two objects appearing as strands in a string diagram, each of which is either one of the $c_i$'s or its dual, a directed odd chord between them means `insert $\iota$ at the target strand' minus `insert $\iota$ at the source strand'. For example:
    \begin{equation*}
        \vcenter{\hbox{\includegraphics[scale=1]{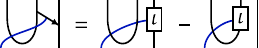}}}\;.
    \end{equation*}
\end{notation}
Note that the endpoints of a chord can be slid along strands (including around evaluation and coevaluation), and that reversing the direction of a chord costs a sign. Furthermore, $\widetilde{h}_a^\vee$ is nothing but the cokernel of the map $\pi \otimes h_a^\vee \to h_a^\vee$ given by
\begin{equation*}
    \vcenter{\hbox{\includegraphics[scale=1]{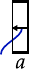}}}\;.
\end{equation*}
Appropriately, Property \ref{itm:005} (which says that $\phi$ factors through $\pi \otimes h^\vee \otimes h^\vee \to \pi \otimes \widetilde{h}^\vee \otimes h^\vee$) simply says that the precomposition of the $ab$'th component of $\phi$ by 
\begin{equation*}
    \includegraphics[scale=1]{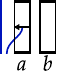}
\end{equation*}
vanishes.

\begin{lem}\label{lem:double-chord}
    For all arrows $a$ in $\overline{Q}$,
    \begin{equation*}
        \includegraphics[scale=1]{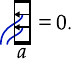}
    \end{equation*}
\end{lem}
\begin{proof}
    Writing $i = s(a)$ and $j=t(a)$, we have
    \begin{equation*}
        \includegraphics[scale=1]{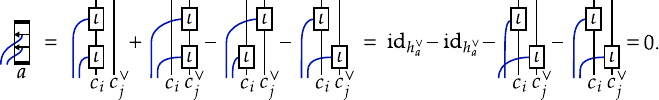}
    \end{equation*}
\end{proof}

Equipped with the chord notation and Lemma \ref{lem:double-chord}, we can show that the maps $\phi^\pm:\pi \otimes h^\vee \otimes h^\vee \to S\widetilde{h}^\vee$ defined below satisfy Properties \ref{itm:006} and \ref{itm:005}. Define $\phi^-$ to have $ab$'th component
\begin{equation*}
    \vcenter{\hbox{\includegraphics[scale=1]{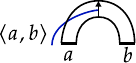}}}\;,
\end{equation*}
and $\phi^+$ to have $ab$'th component
\begin{equation*}
    \vcenter{\hbox{\includegraphics[scale=1]{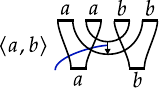}}}\;.
\end{equation*}

\begin{rem}
    These diagrams are essentially the same as the defining diagrams of the Poisson brackets $\{-,-\}^\pm$ --- the only difference is that we have added a chord.
\end{rem}
\begin{prop}
    The maps $\phi^\pm:\pi \otimes h^\vee \otimes h^\vee \to S\widetilde{h}^\vee$ uniquely factor through $\pi \otimes h^\vee \otimes h^\vee \to \pi \otimes S^2 \widetilde{h}^\vee$, and hence we can view them as maps $\pi \otimes S^2 \widetilde{h}^\vee \to S \widetilde{h}^\vee$.
\end{prop}
\begin{proof}
    We need to show that $\phi^-$ and $\phi^+$ satisfy Properties \ref{itm:006} and \ref{itm:005}. Firstly, $\phi^-$:
    \begin{enumerate}
        \item It suffices to show that $\phi^-_{ba} \circ (\mathrm{id}_\pi \otimes \sigma_{h_a^\vee,h_b^\vee}) = \phi^-_{ab}$, where the subscripts denote the components of $\phi^-$. Indeed,
        \begin{equation*}
            \vcenter{\hbox{\includegraphics[scale=1]{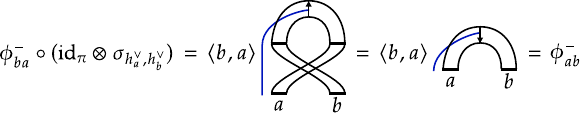}}}\;.
        \end{equation*}

        \item Applying Lemma \ref{lem:double-chord} we have
        \begin{equation*}
            \vcenter{\hbox{\includegraphics[scale=1]{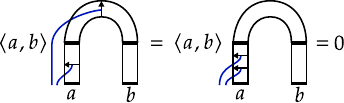}}}\;.
        \end{equation*}
    \end{enumerate}
    And now $\phi^+$:
    \begin{enumerate}
        \item As above, we find that
        \begin{equation*}
            \vcenter{\hbox{\includegraphics[scale=1]{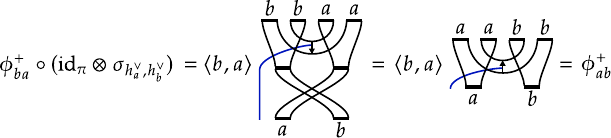}}}\;,
        \end{equation*}
        where we were free to permute the tensor factors at the top since we implicitly map to $S\widetilde{h}^\vee$.

        \item Note that
        \begin{equation*}
            \vcenter{\hbox{\includegraphics[scale=1]{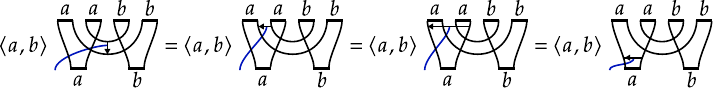}}}\;,
        \end{equation*}
        where in the second equality we used that we implicitly apply $h_a^\vee \to \widetilde{h}_a^\vee$ at the top. Thus, Property \ref{itm:005} holds as a direct consequence of Lemma \ref{lem:double-chord}.
    \end{enumerate}
\end{proof}

Denote by $\widetilde{\Delta}^\pm \in \underline{\mathrm{D}}_2(S\widetilde{h}^\vee)$ the odd 2nd order differential operators obtained by extending $\phi^\pm : \pi \otimes S^2 \widetilde{h}^\vee \to S \widetilde{h}^\vee$ --- in particular, they vanish on $S^{\leq 1} \widetilde{h}^\vee$. Explicitly, formula \eqref{eqn:explicit-graded-diffop} says that on $(h^\vee)^{\otimes N}$ they are given by
\begin{equation}\label{eqn:008}
    \vcenter{\hbox{\includegraphics[scale=1]{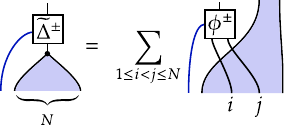}}}\;,
\end{equation}
where the dot denotes $(h^\vee)^{\otimes N} \twoheadrightarrow S^N \widetilde{h}^\vee$, and on the right hand side we implicitly map to $S \widetilde{h}^\vee$ at the top.

\begin{prop}\label{prop:BV-ops-on-Shtildevee}
    The differential operators $\widetilde{\Delta}^\pm$ are \textit{BV} operators. Thus, $(S\widetilde{h}^\vee,\widetilde{\Delta}^\pm)$ is a BV algebra in $\mathcal{C}$.
\end{prop}
\begin{proof}
    We already have that $\widetilde{\Delta}^\pm \circ u = 0$, so it remains to show that $(\widetilde{\Delta}^\pm)^2 = 0$. By Proposition \ref{prop:graded-comm-of-diffops} we have $(\widetilde{\Delta}^\pm)^2 = \frac{1}{2}[\widetilde{\Delta}^\pm,\widetilde{\Delta}^\pm] \in \underline{\mathrm{D}}_3(S\widetilde{h}^\vee)$,\footnote{We are using that $2 \in R$ is invertible in this argument, but the result holds generally: if $\Delta$ is order $k$ and odd, then $\Delta^2$ is order $2k - 1$.} so it suffices to show that $(\widetilde{\Delta}^\pm)^2$ vanishes on $S^{\leq 3}\widetilde{h}^\vee$. Since $\widetilde{\Delta}^-$ vanishes on $S^{\leq 1} \widetilde{h}^\vee$ and has polynomial degree $-2$, it is clear that $(\widetilde{\Delta}^-)^2 = 0$ by degree reasons. We now consider $\widetilde{\Delta}^+$.
    \begin{itemize}
        \item $(\widetilde{\Delta}^+)^2$ vanishes on $S^{\leq 1}\widetilde{h}^\vee$.
        \item To show that $(\widetilde{\Delta}^+)^2$ vanishes on $S^2 \widetilde{h}^\vee$, consider it on $h_a^\vee \otimes h_b^\vee$ (for all arrows $a,b$):
        \begin{align*}
            &\includegraphics[scale=1]{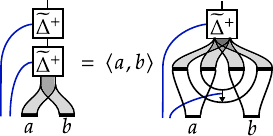} \\
            &\includegraphics[scale=1]{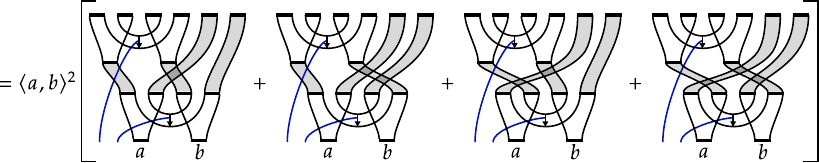} \\
            &\vcenter{\hbox{\includegraphics[scale=1]{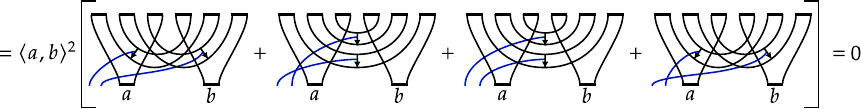}}} \;.
        \end{align*}
        \item For $S^3 \widetilde{h}^\vee$, consider $(\widetilde{\Delta}^+)^2$ on $h_a^\vee \otimes h_b^\vee \otimes h_c^\vee$ (for all arrows $a,b,c$), for which the calculation is similar to that in the proof of Proposition \ref{prop:quartic-poisson-bracket}:
        \begin{equation}\label{eqn:007}
            \vcenter{\hbox{\includegraphics[scale=1]{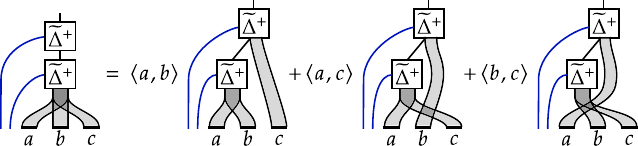}}}\;.
        \end{equation}
        Calculating the first term, we obtain
        \begin{align*}
            \includegraphics[scale=1]{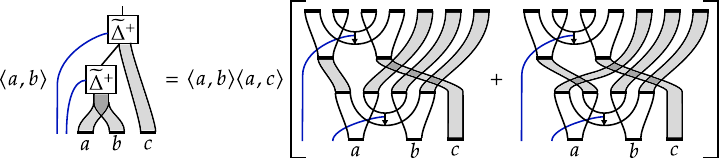}\\
            \vcenter{\hbox{\includegraphics[scale=1]{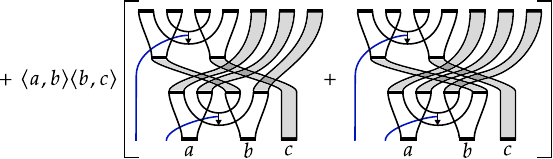}}}\;,
        \end{align*}
        where we used that $(\widetilde{\Delta}^\pm)^2$ vanishes on $S^2 \widetilde{h}^\vee$. Rewriting, this is 
        \begin{align*}
            \includegraphics[scale=1]{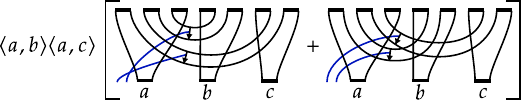}\\
            \vcenter{\hbox{\includegraphics[scale=1]{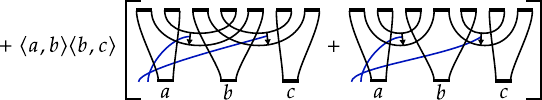}}}\;.
        \end{align*}
        Similarly, we have
        \begin{align*}
            \includegraphics[scale=1]{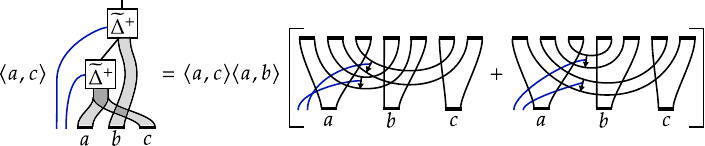}\\
            \includegraphics[scale=1]{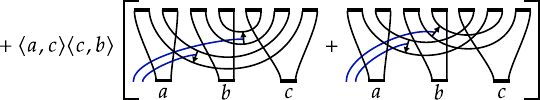}\;
        \end{align*}
        and
        \begin{align*}
            \includegraphics[scale=1]{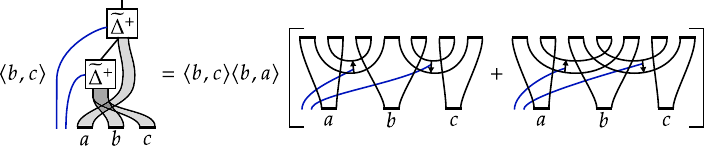}\\
            \vcenter{\hbox{\includegraphics[scale=1]{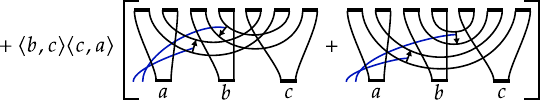}}}\;,
        \end{align*}
        so we see that \eqref{eqn:007} vanishes by antisymmetry of $\langle -,- \rangle$, the sign cost of reversing chords, and $\sigma_{\pi,\pi} = -\mathrm{id}$.
    \end{itemize}
\end{proof}

To relate the representation BV algebra $(S\widetilde{h}^\vee,\widetilde{\Delta}^\pm)$ in $\mathcal{C}$ to the (augmented) necklace BV algebra $(S\Pi \mathcal{N},\Delta^\pm)$ in $\mathrm{smod}_R$, we first apply the lax symmetric monoidal $\Pi$-functor $\underline{\mathcal{C}}(1,-):\mathcal{C} \to \mathrm{smod}_R$ to $(S\widetilde{h}^\vee,\widetilde{\Delta}^\pm)$, yielding an $R$-BV algebra $\underline{\mathcal{C}}(1,S\widetilde{h}^\vee)$ by Proposition \ref{prop:pi-functor-on-BV-algs} (whose BV operator we continue to denote by $\widetilde{\Delta}^\pm$).

Let $\mathrm{otr}:S\Pi \mathcal{N} \to \underline{\mathcal{C}}(1,S\widetilde{h}^\vee)$ be the map of commutative $R$-superalgebras defined on the generating set of cyclic paths by
\begin{equation*}
    \vcenter{\hbox{\includegraphics[scale=1]{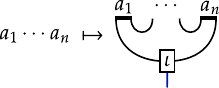}}} \quad \text{and} \quad
    \vcenter{\hbox{\includegraphics[scale=1]{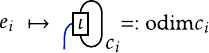}}}\;.
\end{equation*}
Notice that this map, called `odd trace', is well defined --- it is invariant under cyclic permutation of paths\footnote{One now sees why it is crucial to take the coequaliser $h^\vee \twoheadrightarrow \widetilde{h}^\vee$.} and is even, since we are viewing each path as an \textit{odd} element of $\Pi \mathcal{N}$ thanks to the shift.
\begin{thm}\label{thm:intertwining-BV-ops}
    The map $\mathrm{otr}:S\Pi \mathcal{N} \to \underline{\mathcal{C}}(1,S\widetilde{h}^\vee)$ intertwines the BV operators $\Delta^\pm$ and $\widetilde{\Delta}^\pm$.
\end{thm}
\begin{proof}
    Because of the way the BV operator $\Delta^\pm = 2 \mathrm{br}^\pm - \delta^\pm$ on $S\Pi \mathcal{N}$ is constructed, it suffices to show that the following two equations hold for all $x,y \in \Pi \mathcal{N}$:
    \begin{equation}\label{eqn:009}
        \{\mathrm{otr}(x) , \mathrm{otr}(y)\}_{\widetilde{\Delta}^\pm} = \mathrm{otr}(2\mathrm{br}^\pm(x \cdot y)), \; \text{and}
    \end{equation}
    \begin{equation}\label{eqn:010}
        \widetilde{\Delta}^{\pm}(\mathrm{otr}(x)) = \mathrm{otr}(-\delta^\pm(x)),
    \end{equation}
    where $\{- , -\}_{\widetilde{\Delta}^\pm}$ is bracket associated to $\widetilde{\Delta}^\pm$. We only prove the `$-$' case: the `$+$' case is similar. First, we show Equation \eqref{eqn:009}. Given cyclic paths $a_1\cdots a_m$ and $b_1 \cdots b_n$ and using equation \eqref{eqn:008}, the bracket $\{\mathrm{otr}(a_1\cdots a_m) , \mathrm{otr}(b_1 \cdots b_n)\}_{\widetilde{\Delta}^-}$ can be expressed as
    \begin{equation*}
        \vcenter{\hbox{\includegraphics[scale=1]{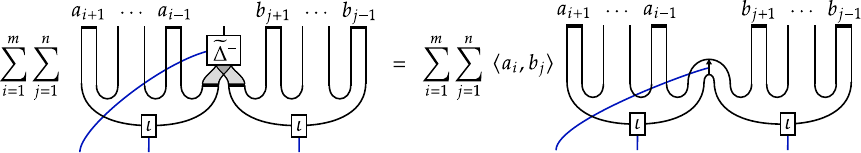}}},
    \end{equation*}
    where we implicitly map into $S\widetilde{h}^\vee$ at the top, as usual. Spelling out the chord, this is nothing but
    \begin{equation*}
        \includegraphics[scale=1]{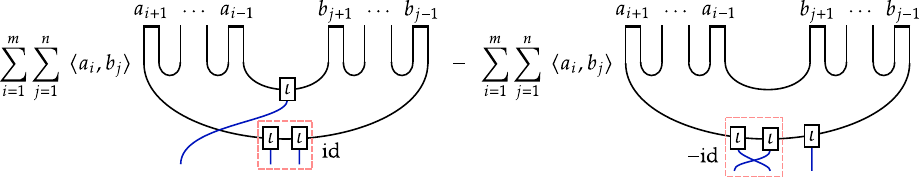}
    \end{equation*}
    which, upon simplifying as indicated by the red boxes, evaluates to
    \begin{equation*}
        \vcenter{\hbox{\includegraphics[scale=1]{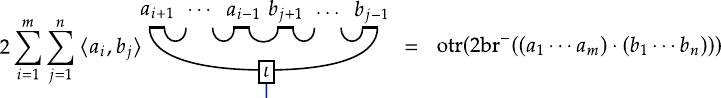}}}\;.
    \end{equation*}
    Next, Equation \eqref{eqn:010}. Given a cyclic path $a_1 \cdots a_n$ and applying equation \eqref{eqn:008} once again, $\widetilde{\Delta}^-(\mathrm{otr}(a_1 \cdots a_n))$ is given by
    \begin{equation*}
        \vcenter{\hbox{\includegraphics[scale=1]{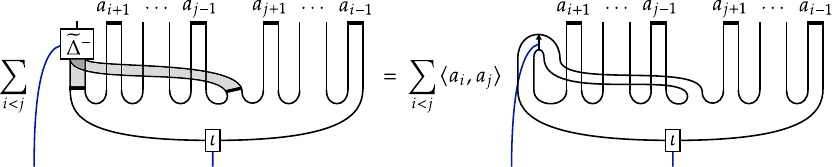}}}\;,
    \end{equation*}
    which is simplified to
    \begin{equation}\label{eqn:011}
        \vcenter{\hbox{\includegraphics[scale=1]{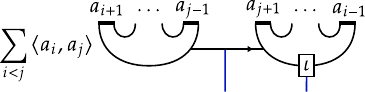}}}\;.
    \end{equation}
    Notice that 
    \begin{equation*}
        \vcenter{\hbox{\includegraphics[scale=1]{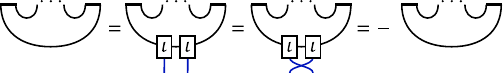}}}\;,
    \end{equation*}
    where in the second equality we used that we implicitly map to $S\widetilde{h}^\vee$ at the top. As $2 \in R$ is invertible, it follows that such terms vanish. Thus, \eqref{eqn:011} is nothing but
    \begin{equation*}
        \vcenter{\hbox{\includegraphics[scale=1]{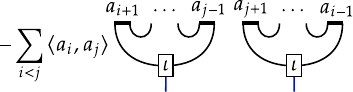}}}\;,
    \end{equation*}
    which is simply
    \begin{align*}
        \mathrm{otr}\left(-\sum_{i<j}\langle a_i, a_j \rangle(a_{i+1}\cdots a_{j-1}) \cdot (a_{j+1}\cdots a_{i-1})\right)
        &=
        \mathrm{otr}\left(-\frac{1}{2}\sum_{i,j}\langle a_i, a_j \rangle (a_{i+1}\cdots a_{j-1}) \cdot (a_{j+1}\cdots a_{i-1})\right)\\
        &=
        \mathrm{otr}(-\delta^-(a_1\cdots a_n)),
    \end{align*}
    completing the proof.
\end{proof}

\begin{expl}
    Suppose $R=k$ is a field (characteristic not $2$) and take $\mathcal{C} = \mathrm{svect}_k$. Each $c_i$ is then a finite dimensional $k$-superspace. The existence of the odd involution $\iota_i$ means that $\mathrm{sdim}(c_i) = 0$. We have that $\underline{\mathcal{C}}(1,S\widetilde{h}^\vee) \cong \mathscr{O}[\mathrm{Rep}_{\iota,\overline{Q}}(c)]$ is the superalgebra of (formal) polynomial functions on the superspace $\mathrm{Rep}_{\iota,\overline{Q}}(c) \subset \mathrm{Rep}_{\overline{Q}}(c)$ of $k$-superlinear representations of $\overline{Q}$ that intertwine the involutions $\iota$. The `$-$' case recovers the BV algebra of \cite{perry2024graded}, and $\mathrm{otr}(a_1\cdots a_m)$ can be interpreted as the function sending a representation $f$ to $\mathrm{str}(f_{a_m} \circ \cdots \circ f_{a_1} \circ \iota)$.
\end{expl}

We suspect that $\mathrm{otr}$ can be injective, so that the full necklace IBL structure is encoded in our categorical construction. We address this question in \S\ref{sec:Injectivity-of-otr}.

\section{The intrinsically graded variant}\label{sec:grNLBs}
In \S\ref{subsec:lie-bialg-structure} we needed to perform a shift $\mathcal{N} \mapsto \Pi \mathcal{N}$ in order to make sense of the full IBL structure $(\mathrm{br}^\pm,\delta^\pm)$ from the representation perspective. While on the necklace side this seems like a harmless trick, modelling this shift forced us to make a \textit{choice} of odd involutions (which may not even exist). In this section we show that such choices need not be made, provided we replace $\mathcal{N}$ with a closely-related graded variant `$\mathcal{N}_\text{gr}$'. This is the `$p=1$' case of \cite{perry2024graded}, introduced prior for single-vertex quivers in \cite{barannikov2014matrix}.

Since the shift should be avoided, the natural algebraic structure for the BV story should be the one found on $\Pi \mathcal{N}$, \textbf{not} on $\mathcal{N}$. In the language of \cite{perry2024graded} this structure is a \textit{degree 1 involutive Lie bialgebra}, but we call it an \textit{odd IBL algebra} in this text, for brevity.\footnote{Terminology on such structures has not yet been settled. Appendix B of \cite{perry2024graded} compares related notions.} The supermodule $\mathcal{N}_\text{gr}$ will turn out to have this structure. Let us recall (and generalise) its definition.
\begin{defn}\label{defn:odd-IBL}
    Let $\mathcal{C} \in \mathrm{SymMon}\Pi\text{-}\mathrm{Cat}_R$. An \textbf{odd IBL algebra} in $\mathcal{C}$ is an object $A \in \mathcal{C}$ together with \textit{odd} maps $\mathrm{br}:A \otimes A \to A$ and $\delta:A \to A\otimes A$ that are \textit{symmetric}\footnote{Symmetric simply means that $ \mathrm{br} \circ \sigma = \mathrm{br}$ and $\sigma \circ \delta = \delta$.} and satisfy the following four equations in $\underline{\mathcal{C}}$:
    \begin{gather*}
        \mathrm{br} \circ (\mathrm{br} \otimes \mathrm{id}) \circ \mathfrak{S}_3 = 0, \quad
        \mathfrak{S}_3 \circ (\delta \otimes \mathrm{id}) \circ \delta = 0, \quad
        \mathrm{br} \circ \delta = 0, \quad \text{and}\\
        \delta \circ \mathrm{br} + \mathfrak{S}_2 \circ (\mathrm{br} \otimes \mathrm{id}) \circ (\mathrm{id} \otimes \sigma) \circ (\delta \otimes \mathrm{id}) \circ \mathfrak{S}_2 = 0,
    \end{gather*}
    where $\mathfrak{S}_2 = \mathrm{id} + \sigma$. These are the \textit{Jacobi}, \textit{coJacobi}, \textit{involutivity} and \textit{cocycle} conditions, respectively.
\end{defn}

\begin{rem}
    If $A$ is an odd IBL algebra, then $\Pi A$ is an IBL algebra in the usual sense --- and vice versa. For details on how the maps $\mathrm{br}$ and $\delta$ are shifted, we refer the reader to \cite{perry2024graded}.
\end{rem}
\begin{rem}
    In $\mathrm{smod}_R$, the cocycle condition can be written in the more-familiar form
    \begin{equation*}
        \delta(\mathrm{br}(x , y)) = (-1)^{|x| + 1}\mathrm{ad}_x(\delta(y)) + (-1)^{(|x|+1)|y|+1}\mathrm{ad}_y(\delta(x)).
    \end{equation*}
\end{rem}

The connection between odd IBL algebras and BV algebras is quite direct:
\begin{prop}\label{prop:oddIBL-BV}
    Suppose that $\mathcal{C}$ is $\otimes$-cocomplete, and let $\mathrm{br}:A \otimes A \to A$ and $\delta:A \to A \otimes A$ be odd symmetric maps in $\underline{\mathcal{C}}$. If the odd 2nd order differential operator $\mathrm{br} + \delta$ on $SA$ is a BV operator (i.e.\ it squares to zero), then $(A,\mathrm{br},\delta)$ is an odd IBL algebra. The converse holds if $3 \in R$ is not a zero divisor for the module $\mathcal{C}(A,S^3 A)$.
\end{prop}
\begin{proof}
    This is \cite[Prop.\ 2.3]{perry2024graded} in a more general context, and one simply interprets its proof diagrammatically. The maps `$I_n$' of that proof are here the maps $S^nA \to A^{\otimes n}$ induced by $\sum_{\tau \in S_n} \sigma_\tau:A^{\otimes n}\to A^{\otimes n}$. However, these are not in general monic. In addressing this problem, the forward implication uses that $2$ is not a zero divisor for the module $\mathcal{C}(A,A^{\otimes 3})$, nor for the module $\mathcal{C}(A^{\otimes 2},A^{\otimes2})$ --- which is the case as $2$ is assumed invertible. The backwards implication does not use any assumptions on $2$, but uses the assumption on $3$ specified in the claim.
\end{proof}

Denote by $(R\overline{Q})_{\text{gr}}$ the $R$-superalgebra\footnote{For clarity: we mean one that is unital and associative, that is, a monoid in $\mathrm{smod}_R$.} whose underlying $R$-algebra is $R \overline{Q}$ and whose $\mathbb{Z}_2$-grading is the one induced by declaring $|a| = 0$ and $|\overline{a}| = 1$ for all $a \in Q_1$. In particular, $|e_i| = 0$ and $|a_1\cdots a_n| = \#\{i\;|\; a_i \notin Q_1\}$. Let $\mathcal{N}_\text{gr} := (R\overline{Q})_{\text{gr}} / [(R\overline{Q})_{\text{gr}},(R\overline{Q})_{\text{gr}}]$ be the quotient of $(R\overline{Q})_{\text{gr}}$ by its \textit{graded} commutator. The imposed relations on paths are
\begin{equation*}
    a_1 \cdots a_n = (-1)^{|a_1||a_2\cdots a_n|}a_2 \cdots a_n a_1 = \dots = (-1)^{|a_1 \cdots a_{n-1}||a_n|}a_na_1\cdots a_{n-1}.
\end{equation*}
Notice that while $\mathcal{N}_\text{gr}$ is spanned by closed paths, a closed path may well be zero as a consequence of the relations. Unlike in $\mathcal{N}$, endpoints now matter (up to a sign). Nevertheless, $\mathcal{N}_\text{gr}$ is free: a set of closed paths that are \textit{nonzero} in $\mathcal{N}_{\text{gr}}$ is linearly independent in $\mathcal{N}_{\text{gr}}$ if and only if the underlying cyclic paths are pairwise distinct, and letting $\mathscr{B}$ be any maximal such set gives an isomorphism $\mathcal{N}_\text{gr} \cong \langle \mathscr{B}\rangle_R$ as modules.

We will show that $\mathcal{N}_\text{gr}$ carries two canonical odd IBL structures, $(\mathrm{br}^-_\text{gr},\delta^-_\text{gr})$ and $(\mathrm{br}^+_\text{gr},\delta^+_\text{gr})$ --- graded analogues of what we've already seen. The `$-$' case was already introduced in \cite{perry2024graded}, called the \textit{graded necklace Lie bialgebra}. We will derive the structure maps $(\mathrm{br}^\pm_\text{gr},\delta^\pm_\text{gr})$ \textit{from the representation side}, and prove that they satisfy the odd IBL conditions from there. This turns out to simplify matters greatly.

Fix a $\otimes$-cocomplete symmetric monoidal $\Pi$-category $\mathcal{C}$. For each vertex $i$ fix a dualisable object $c_i \in \mathcal{C}$. Put $h^\vee_{\text{gr}} = \bigoplus_a \left(\pi^{|a|} \otimes h_a^\vee\right)$, where $h_a^\vee = c_{s(a)} \otimes c_{t(a)}^\vee$ as before, and consider the free commutative monoid $Sh_\text{gr}^\vee$. Define $\mathrm{tr}_{\text{gr}}:
S\mathcal{N}_\text{gr} \to \underline{\mathcal{C}}(1,Sh_\text{gr}^\vee)$ to be the map of commutative $R$-superalgebras which sends
\begin{equation*}
    \vcenter{\hbox{\includegraphics[scale=1]{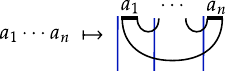}}} \quad \text{and} \quad
    \vcenter{\hbox{\includegraphics[scale=1]{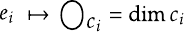}}}\;,
\end{equation*}
where each blue strand denotes either $1$ or $\pi$, as appropriate. This is well defined.

The monoid $Sh_\text{gr}^\vee$ admits two natural BV operators. To describe them, consider first the canonical \textit{symmetric} pairing $\langle -,- \rangle_\text{gr}:\overline{Q}_1 \times \overline{Q}_1 \to \mathbb{Z}$ given by
\begin{equation*}
    \langle a,b \rangle_\text{gr} =
    \begin{cases}
        1, \; \text{if $a = \overline{b}$ or $b = \overline{a}$,}\\
        0, \; \text{otherwise.}
    \end{cases}
\end{equation*}
Define maps $\psi^\pm:\pi \otimes h_\text{gr}^\vee \otimes h_\text{gr}^\vee \to Sh_\text{gr}^\vee$ in $\mathcal{C}$ as follows: $\psi^-$ has $ab$'th component
\begin{equation*}
    \includegraphics[scale=1]{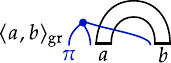}\;
\end{equation*}
and $\psi^+$ has $ab$'th component
\begin{equation*}
    \vcenter{\hbox{\includegraphics[scale=1]{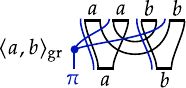}}}\;.
\end{equation*}
Since these maps coequalise the $S_2$-action on $h_\text{gr}^\vee \otimes h_\text{gr}^\vee$, they can be viewed as odd maps $S^2 h_\text{gr}^\vee \to Sh_\text{gr}^\vee$ in $\underline{\mathcal{C}}$. Let $\nabla^\pm \in \underline{\mathrm{D}}_2(Sh_\text{gr}^\vee)$ be the odd 2nd order differential operators obtained by extending $\psi^\pm$ --- in particular, they vanish on $S ^{\leq 1} h_\text{gr}^\vee$.
\begin{prop}
    The differential operators $\nabla^\pm$ are BV operators. Thus, $(Sh_\text{gr}^\vee,\nabla^\pm)$ is a BV algebra in $\mathcal{C}$.
\end{prop}
\begin{proof}
    We need to show that the 3rd order differential operators $(\nabla^\pm)^2$ vanish. Arguing as in the proof of Proposition \ref{prop:BV-ops-on-Shtildevee}, the only work is to show that $(\nabla^+)^2$ vanishes on $S^2h_\text{gr}^\vee$ and $S^3 h_\text{gr}^\vee$. For $S^2h_\text{gr}^\vee$, consider $(\nabla^+)^2$ on $(\pi^{|a|} \otimes h_a^\vee) \otimes (\pi^{|b|}\otimes h_b^\vee)$: when $\langle a,b\rangle_\text{gr} = 0$ it immediately vanishes, so it remains to show that $(\nabla^+)^2$ vanishes on $(\pi \otimes h_{\overline{a}}^\vee) \otimes h_a^\vee$ for all $a \in Q_1$.\footnote{We reduce to this case here, unlike in the proof of Proposition \ref{prop:BV-ops-on-Shtildevee}, so that the diagrams are not too convoluted.} Calculating, the map $(\pi \otimes h_{\overline{a}}^\vee) \otimes h_a^\vee \to Sh_\text{gr}^\vee \stackrel{\nabla^+}{\to} Sh_\text{gr}^\vee$ is
    \begin{equation*}
        \vcenter{\hbox{\includegraphics[scale=1]{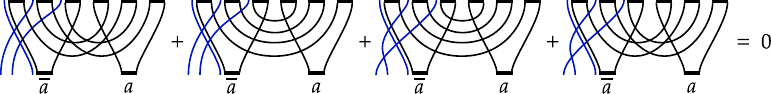}}}.
    \end{equation*}
    Therefore $(\nabla^+)^2$ vanishes on $S^2h_\text{gr}^\vee$. For $S^3 h_\text{gr}^\vee$ one follows the calculation in the proof of Proposition \ref{prop:BV-ops-on-Shtildevee}, modified appropriately --- as such, we will not spell it out. We point out that one can reduce the problem to showing that $(\nabla^+)^2$ vanishes on $(\pi \otimes h_{\overline{a}}^\vee) \otimes h_a^\vee \otimes h_a^\vee$ and $(\pi \otimes h_{\overline{a}}^\vee) \otimes (\pi \otimes h_{\overline{a}}^\vee) \otimes h_a^\vee$, for all $a \in Q_1$, which makes the diagrammatics more manageable.
\end{proof}

Thus, $\underline{\mathcal{C}}(1,Sh_\text{gr}^\vee)$ is an $R$-BV algebra. By way of $\mathrm{tr}_\text{gr}$ we can witness the BV operators $\nabla^\pm$ from the point of view of $\mathcal{N}_\text{gr}$. We first consider $\nabla^-$. The quantity $\{\mathrm{tr}_\text{gr}(a_1\cdots a_m),\mathrm{tr}_\text{gr}(b_1 \cdots b_n)\}_{\nabla^-}$ is computed to be
\begin{align*}
    \includegraphics[scale=1]{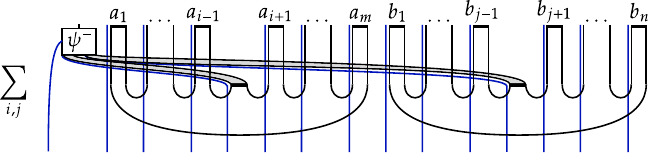}\\
    \vcenter{\hbox{\includegraphics[scale=1]{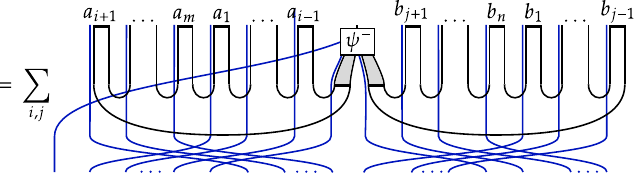}}}\;.
\end{align*}
Using $\sigma_{\pi,\pi} = -\mathrm{id}$ and the definition of $\psi^-$, this is nothing but
\begin{equation*}
    \vcenter{\hbox{\includegraphics[scale=1]{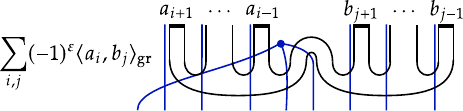}}}
    = \;
    \vcenter{\hbox{\includegraphics[scale=1]{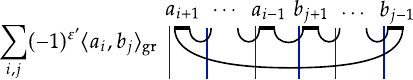}}}\;,
\end{equation*}
where here $\varepsilon = |a_1 \cdots a_i||a_{i+1}\cdots a_m| + |b_1 \cdots b_{j-1}||b_j\cdots b_n|$ and $\varepsilon' = \varepsilon + |a_{i+1} \cdots a_{i-1}|$. Thus, if we define an odd map $\mathrm{br}^-_\text{gr}:\mathcal{N}_\text{gr} \otimes \mathcal{N}_\text{gr} \to \mathcal{N}_\text{gr}$ by
\begin{equation*}
    \mathrm{br}^-_\text{gr}(a_1\cdots a_m \otimes b_1 \cdots b_n) := \sum_{i,j} \pm \langle a_i,b_j \rangle_{\text{gr}} \, e_{t(a_i)}a_{i+1} \cdots a_{i-1} b_{j+1} \cdots b_{j-1},
\end{equation*}
where the $\pm$ sign has parity $|a_1\cdots a_{i-1}||a_i \cdots a_m| + |b_1 \cdots b_{j-1}||b_j\cdots b_n| + |a_{i+1} \cdots a_{i-1}||b_j|$\footnote{This agrees with $\varepsilon'$ when $\langle a_i,b_j \rangle_{\text{gr}} \neq 0$, rewritten using $|a_i| + |b_j| = 1$.} --- which is well-defined and symmetric (as can be easily checked) --- then
\begin{equation}\label{eqn:014}
    \{\mathrm{tr}_\text{gr}(x),\mathrm{tr}_\text{gr}(y)\}_{\nabla^-} = \mathrm{tr}_\text{gr}
    (\mathrm{br}^-_{\text{gr}}(x \cdot y)) \quad \text{for all $x,y \in \mathcal{N}_{\text{gr}}$}.
\end{equation}

Similarly, we can calculate $\nabla^-(\mathrm{tr}_\text{gr}(a_1\cdots a_n))$, which comes out to be
\begin{equation*}
    \vcenter{\hbox{\includegraphics[scale=1]{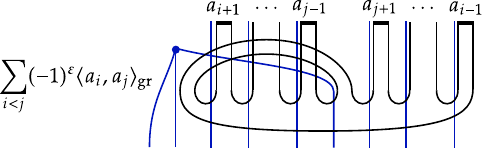}}}
    =
    \vcenter{\hbox{\includegraphics[scale=1]{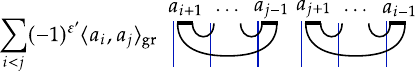}}}\;,
\end{equation*}
where here $\varepsilon = |a_1 \cdots a_{i-1}||a_i \cdots a_n|$ and $\varepsilon' = \varepsilon + |a_{i+1} \cdots a_{j-1}||a_j|$. Thus, defining an odd map $\delta_{\text{gr}}^-:\mathcal{N}_{\text{gr}} \to \mathcal{N}_{\text{gr}} \otimes \mathcal{N}_{\text{gr}}$ by
\begin{equation*}
    \delta^-_\text{gr}(a_1 \cdots a_n) := \frac{1}{2}\sum_{i,j}\pm \langle a_i, a_j \rangle_\text{gr} \, e_{t(a_i)}a_{i+1} \cdots a_{j-1} \otimes e_{t(a_j)}a_{j+1} \cdots a_{i-1},
\end{equation*}
where parity of the $\pm$ sign is $|a_1 \cdots a_{i-1}||a_i \cdots a_n| + |a_{i+1} \cdots a_{j-1}||a_j|$ --- which is again well-defined and symmetric --- yields
\begin{equation}\label{eqn:015}
    \nabla^-(\mathrm{tr}_\text{gr}(x)) = \mathrm{tr}_\text{gr}(\delta^-_\text{gr}(x)) \quad \text{for all $x \in \mathcal{N}_\text{gr}$}.
\end{equation}

As for $\nabla^+$, proceeding the same way leads us to consider symmetric odd maps $\mathrm{br}^+_\text{gr}:\mathcal{N}_\text{gr} \otimes \mathcal{N}_\text{gr} \to \mathcal{N}_\text{gr}$ and $\delta^+_{\text{gr}}:\mathcal{N}_\text{gr} \to \mathcal{N}_\text{gr} \otimes \mathcal{N}_\text{gr}$ defined by
\begin{align*}
    \mathrm{br}^+_\text{gr}(a_1\cdots a_m \otimes b_1 \cdots b_n) := \sum_{i,j} (-1)^\varepsilon \langle a_i,b_j \rangle_{\text{gr}} \, a_ia_{i+1} \cdots a_{i-1} a_ib_j b_{j+1} \cdots b_{j-1}b_j, \; \text{with}\\
    \varepsilon = |a_1\cdots a_{i-1}||a_i \cdots a_m| + |b_1\cdots b_{j-1}||b_j \cdots b_n| + |a_{i+1}\cdots a_{i-1}| + |b_{j+1}\cdots b_{j-1}||b_j|,
\end{align*}
and
\begin{align*}
    \delta^+_\text{gr}(a_1 \cdots a_n) := \frac{1}{2}\sum_{i,j}(-1)^\varepsilon \langle a_i, a_j \rangle_\text{gr} \, a_ia_{i+1} \cdots a_{j-1} a_j \otimes a_ja_{j+1} \cdots a_{i-1}a_i, \; \text{with}\\
    \varepsilon = |a_1\cdots a_{i-1}||a_i\cdots a_n| + |a_{i+1} \cdots a_{j-1}| + |a_{j+1}\cdots a_{i-1}||a_i| + |a_j|.
\end{align*}
As usual, these maps are defined to vanish on constant paths. From their derivation, it is immediate that for all $x,y \in \mathcal{N}_\text{gr}$ one has
\begin{equation}\label{eqn:016}
    \{\mathrm{tr}_\text{gr}(x),\mathrm{tr}_\text{gr}(y)\}_{\nabla^+} = \mathrm{tr}_\text{gr}
    (\mathrm{br}^+_{\text{gr}}(x \cdot y)) \quad \text{and} \quad
    \nabla^+(\mathrm{tr}_\text{gr}(x)) = \mathrm{tr}_\text{gr}(\delta^+_\text{gr}(x)).
\end{equation}

Extending $\mathrm{br}^\pm_\text{gr}$ and $\delta^\pm_\text{gr}$ to odd differential operators on $S\mathcal{N}_\text{gr}$ in the usual way, consider the odd 2nd order differential operators $\Delta_{\text{gr}}^\pm := \mathrm{br}^\pm_\text{gr} + \delta^\pm_{\text{gr}}$. Equations \eqref{eqn:014}, \eqref{eqn:015} and \eqref{eqn:016} imply the following:
\begin{prop}\label{prop:intertwining-of-BV-ops-graded}
    The map $\mathrm{tr}_\text{gr}:S\mathcal{N}_\text{gr} \to \underline{\mathcal{C}}(1,Sh_\text{gr}^\vee)$ intertwines $\Delta_\text{gr}^\pm$ with the BV operator $\nabla^\pm$.
\end{prop}

\begin{thm}\label{thm:graded-necklace-BV-alg}
    $(S\mathcal{N}_\text{gr},\Delta^\pm_\text{gr})$ is a BV algebra.
\end{thm}
\begin{proof}
    We need to show that $(\Delta^\pm_\text{gr})^2 = 0$. By Proposition \ref{prop:intertwining-of-BV-ops-graded}, it suffices to show that there exists a choice of $\mathcal{C}$ and dualisable objects $c_i$ for which $\mathrm{tr}_\text{gr}$ is injective.

    Here we maintain the notation from the proof of Proposition \ref{prop:injectivity-of-tr}. Applying Construction \ref{constr:addmon-to-pimon} to $\widehat{\mathcal{A}}$ gives a cocomplete closed symmetric monoidal $\Pi$-category $\mathcal{S}\widehat{\mathcal{A}}$. Take $\mathcal{C} = \mathcal{S}\widehat{\mathcal{A}}$ and let $c_i = (\yo(i),0) \in \mathcal{C}$ (which is symmetrically self-dual) for each vertex $i$, so that $h_a^\vee \cong (\yo(a),0)$. Notice that the $S_n$-action on
    \begin{equation*}
        (h_\text{gr}^\vee)^{\otimes n} \cong \bigoplus_{a_1,\dots,a_n}\left((\pi^{|a_1|} \otimes h_{a_1}^\vee) \otimes \cdots \otimes (\pi^{|a_n|} \otimes h_{a_n}^\vee) \right)
        \cong
        \bigoplus_{a_1,\dots,a_n}\left(\pi^{|a_1\cdots a_n|} \otimes (\yo(a_1,\dots,a_n),0)\right)
    \end{equation*}
    is described as follows: a permutation $\tau \in S_n$ acts by the maps
    \begin{equation}\label{eqn:S_n-action}
    \mathrm{id} \otimes(\yo(\pm\tau \cdot),0):
        \pi^{|a_1\cdots a_n|} \otimes (\yo(a_1,\dots,a_n),0) \to 
        \pi^{|a_1\cdots a_n|} \otimes (\yo(a_{\tau(1)},\dots,a_{\tau(n)}),0),
    \end{equation}
    where $\pm$ is the appropriate Koszul sign and `$\tau \cdot$' is the action via the symmetry of $\mathrm{Diag_{\overline{Q}}}$. Using similar arguments to those in the proof of Proposition \ref{prop:injectivity-of-tr}, one finds that
    \begin{equation*}
        \underline{\mathcal{C}}(1,Sh_\text{gr}^\vee) 
        \cong \bigsqcup_n \left(\bigoplus_{a_1,\dots,a_n}\mathcal{A}(\emptyset, (a_1,\dots,a_n))\right)_{S_n},
    \end{equation*}
    where $\mathcal{A}(\emptyset, (a_1,\dots,a_n))$ lies in the degree $|a_1\cdots a_n|$ part of the supermodule, and where colimits are computed in $\mathrm{smod}_R$. By \eqref{eqn:S_n-action}, the $S_n$-action is seen to be the one induced by post-composition using the symmetry of $\mathrm{Diag_{\overline{Q}}}$, adjusted by the appropriate sign, for example,
    \begin{equation}\label{eqn:017}
        \vcenter{\hbox{\includegraphics[scale=1]{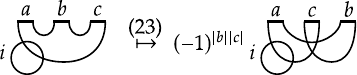}}}\;.
    \end{equation}
    Rewriting this slightly differently, we have
    \begin{equation}\label{eqn:022}
        \underline{\mathcal{C}}(1,Sh_\text{gr}^\vee) \cong 
        \left\langle \bigsqcup_n\bigsqcup_{a_1,\dots,a_n}\mathrm{Diag}_{\overline{Q}}(\emptyset,(a_1,\dots,a_n))\right\rangle_R \Big/ \sim,
    \end{equation}
    the free module on the set of coloured string diagrams from $\emptyset$ to finite lists of arrows $(a_1,\dots,a_n)$, equipped with the aformentioned grading, modulo relations such as \eqref{eqn:017}.

    Under \eqref{eqn:022} the map $\mathrm{tr}_\text{gr}$ sends $(e_{i_1})\cdots (e_{i_k}) \cdot (a_{1}^{1}\cdots a_{n_1}^{1}) \cdots (a_{1}^{p}\cdots a_{n_p}^{p}) \in S\mathcal{N}_\text{gr}$ to (the equivalence class of)
    \begin{equation*}
        \vcenter{\hbox{\includegraphics[scale=1]{Diagrams/093.pdf}}}\;.
    \end{equation*}
    One can already convince themselves that $\mathrm{tr}_\text{gr}$ is injective, but let us make the argument precise. Recall that $\mathcal{N}_\text{gr} \cong \langle \mathscr{B} \rangle_R$ is free as a module, where $\mathscr{B}$ is a certain set of closed paths. It follows that $S\mathcal{N}_\text{gr} \cong \langle \mathscr{B}' \rangle_R$ is free as a module too, where $\mathscr{B}'$ is any maximal collection of finite tuples of elements of $\mathscr{B}$ that are nonzero in $S\mathcal{N}_\text{gr}$ and whose underlying unordered tuples\footnote{By an `unordered tuple' we simply mean a multiset.} are pairwise distinct. We then observe that the map $\mathscr{B}' \to \underline{\mathcal{C}}(1,Sh_\text{gr}^\vee)$  inducing $\mathrm{tr}_\text{gr}$ --- which is injective by a similar argument to that in the proof of Proposition \ref{prop:injectivity-of-tr} --- has linearly independent image, and so it follows that $\mathrm{tr}_\text{gr}$ is itself injective.
\end{proof}

Applying Proposition \ref{prop:oddIBL-BV} then yields the following, which would be very tedious to verify directly:
\begin{cor}\label{cor:-grneck-is-odd-ibl}
    The triple $(\mathcal{N}_\text{gr},\mathrm{br}^\pm_\text{gr},\delta^\pm_\text{gr})$ is an odd IBL algebra. 
\end{cor}

\section{Injectivity of $\mathrm{otr}$}\label{sec:Injectivity-of-otr}
In this section we address the question of whether the map $\mathrm{otr}:S\Pi\mathcal{N} \to \underline{\mathcal{C}}(1,S\widetilde{h}^\vee)$ can be injective. While we suspect that the answer is \textit{yes}, the difficulty here is that if the $\Pi$-category $\mathcal{C}$ is obtained via Construction \ref{constr:addmon-to-pimon} as (possibly the cocompletion of) $\mathcal{SB}$ for some $\mathcal{B}$, then $\mathrm{otr}$ cannot be injective: the only map $\pi \to 1$ in $\mathcal{C}$ is zero, and hence $\mathrm{otr}(e_i) = 0$. For this reason, taking $\mathcal{C} = \mathcal{S}\widehat{\mathcal{A}}$ as in the proof of Theorem \ref{thm:graded-necklace-BV-alg} does not suffice. However, it almost does.

\begin{prop}\label{prop:near-injectivity-of-otr}
    Let $\mathcal{C} = \mathcal{\mathcal{S}\widehat{\mathcal{A}}}$. The object $(\yo(i),\yo(i)) \in \mathcal{C}$ is dualisable. For each $i$, let $c_i = (\yo(i),\yo(i))$ and let $\iota_i:\pi \otimes c_i \to c_i$ be the obvious isomorphism (which is an odd involution). Then the BV algebra map
    \begin{equation*}
        \mathrm{otr}:S\Pi\mathcal{N} \to \underline{\mathcal{C}}(1,S\widetilde{h}^\vee)
    \end{equation*}
    has kernel $\langle \{e_i\}_{i\in \overline{Q}_0} \rangle$, the ideal generated by constant paths.
\end{prop}
\begin{proof}
    Let $\mathrm{Diag}_{\overline{Q}}'$ be the category whose objects are finite tuples of elements of $\overline{Q}_0 \times \mathbb{Z}_2$ --- thought of as objects of $\mathrm{Diag}_{\overline{Q}}$ where each point has an extra binary label which we call \textit{flavour} --- and whose maps are simply the maps in $\mathrm{Diag}_{\overline{Q}}$ between underlying unflavoured families. For example, the following is a map $(j_0) \to (i_1,j_0,i_0)$ in $\mathrm{Diag}_{\overline{Q}}'$:
    \begin{equation*}
        \vcenter{\hbox{\includegraphics[scale=1]{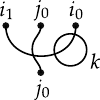}}}\;.
    \end{equation*}
    This category is symmetric monoidal in the natural way, with respect to which forgetting flavour yields a symmetric monoidal equivalence $\mathrm{Diag}_{\overline{Q}}' \simeq \mathrm{Diag}_{\overline{Q}}$. Hence we can view $\mathcal{A}$ as the free $R$-linear SMC on $\mathrm{Diag}_{\overline{Q}}'$, which will turn out to be helpful for keeping track of monoidal factors. Put $c_i = (\yo(i_0),\yo(i_1)) \in \mathcal{C}$. From now on we suppress the notation `$\yo$', leaving the Yoneda embedding $\mathcal{A} \hookrightarrow \widehat{\mathcal{A}}$ implicit.

    We need to show that $c_i = ((i_0),(i_1))$ is dualisable. Let $c_i^\vee = c_i$ and define $\mathrm{ev}:c_i^\vee \otimes c_i \to 1$ to be the map
    \begin{equation*}
        ((i_0,i_0)\oplus(i_1,i_1),(i_0,i_1)\oplus(i_1,i_0)) \to (\emptyset ,0)
    \end{equation*}
    whose nonzero components are
    \begin{equation*}
        \vcenter{\hbox{\includegraphics[scale=1]{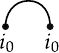}}}\quad 
        \text{and}\quad 
        \vcenter{\hbox{\includegraphics[scale=1]{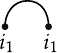}}}.
    \end{equation*}
    Define $\mathrm{coev}:1 \to c_i \otimes c_i^\vee$ similarly. A quick calculation shows that $(c_i,c_i^\vee,\mathrm{ev},\mathrm{coev})$ is duality data. So $c_i$ is dualisable, and in particular self-dual. Caution must be taken however because it is \textit{not symmetrically self-dual}: due to the sign in \eqref{eqn:018} the map $\mathrm{ev} \circ \sigma : c_i \otimes c_i^\vee \to 1$ has nonzero components
    \begin{equation*}
        \vcenter{\hbox{\includegraphics[scale=1]{Diagrams/110.pdf}}}\quad 
        \text{and}\quad 
        \vcenter{\hbox{\includegraphics[scale=1]{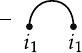}}}.
    \end{equation*}
    If $a$ is an arrow with $s(a) = i$ and $t(a) = j$, then $h_a^\vee = c_i \otimes c_j^\vee = ((i_0,j_0)\oplus(i_1,j_1),(i_0,j_1)\oplus(i_1,j_0))$. When the arrow $a$ is remembered, we denote the pair $(i_\mu,j_\nu)$ by `$a_{\mu\nu}$' --- which we refer to as a \textit{flavoured} arrow --- so that $h_a^\vee = (a_{00}\oplus a_{11} ,a_{01} \oplus a_{10})$.

    Now, the odd involution $\iota_i:\pi \otimes c_i \to c_i$ is the map $(i_1,i_0) \to (i_0,i_1)$ with components
    \begin{equation*}
        \vcenter{\hbox{\includegraphics[scale=1]{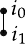}}}\quad 
        \text{and}\quad 
        \vcenter{\hbox{\includegraphics[scale=1]{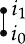}}}\;.
    \end{equation*}
    Its dual map $\iota_i^\vee:\pi \otimes c_i^\vee \to c_i^\vee$ is then the map $(i_1,i_0) \to (i_0,i_1)$ with components
    \begin{equation*}
        \vcenter{\hbox{\includegraphics[scale=1]{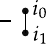}}}\quad 
        \text{and}\quad 
        \vcenter{\hbox{\includegraphics[scale=1]{Diagrams/113.pdf}}}\;,
    \end{equation*}
    where the sign comes from the crossing in \eqref{eqn:020}. Calculating further, the two maps $\pi \otimes h_a^\vee \to h_a^\vee$ shown in \eqref{eqn:019}, which are maps $(a_{01}\oplus a_{10},a_{00}\oplus a_{11}) \to (a_{00}\oplus a_{11},a_{01}\oplus a_{10})$, are found to have components
    \begin{equation}\label{eqn:021}
        \vcenter{\hbox{\includegraphics[scale=1]{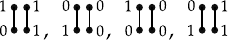}}}\;\quad \text{and}
        \quad \vcenter{\hbox{\includegraphics[scale=1]{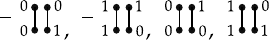}}}\;,
    \end{equation}
    respectively, where we have dropped the colours (keeping only flavours) as they are clear from context. Note that we had to take care of additional signs coming from the crossing in \eqref{eqn:019}.

    Consider the degree $\mu \in \mathbb{Z}_2$ part of the supermodule $\underline{\mathcal{C}}(1,S\widetilde{h}^\vee)$, written $\underline{\mathcal{C}}(1,S\widetilde{h}^\vee)_\mu$, which is identified with $\widehat{\mathcal{A}}(\emptyset,(S\widetilde{h}^\vee)_\mu) \cong (S\widetilde{h}^\vee)_\mu(\emptyset)$. Unpacking,
    \begin{equation*}
        (S\widetilde{h}^\vee)_\mu \cong \bigsqcup_n\left(\bigoplus_{a_1,\dots,a_n}\bigoplus_{\sum \mu_i = \mu}(\widetilde{h}_{a_1}^\vee)^{}_{\mu_1} \otimes \dots \otimes (\widetilde{h}_{a_n}^\vee)^{}_{\mu_n}\right)_{S_n},
    \end{equation*}
    where we used that colimits in $\mathcal{S}\widehat{\mathcal{A}}$ are computed degree-wise, and where $\tau \in S_n$ acts via the maps
    \begin{equation*}
        \pm \tau \cdot: (\widetilde{h}_{a_1}^\vee)^{}_{\mu_1} \otimes \dots \otimes (\widetilde{h}_{a_n}^\vee)^{}_{\mu_n} \isoto (\widetilde{h}_{a_{\tau(1)}}^\vee)^{}_{\mu_{\tau(1)}} \otimes \dots \otimes (\widetilde{h}_{a_{\tau(n)}}^\vee)^{}_{\mu_{\tau(n)}}
    \end{equation*}
    using the symmetry in $\widehat{\mathcal{A}}$, adjusted by the appropriate Koszul sign arising from rearranging the $\mu_i$'s. It follows that
    \begin{equation*}
        (S\widetilde{h}^\vee)_\mu(\emptyset) \cong \bigsqcup_n\left(\bigoplus_{a_1,\dots,a_n}\bigoplus_{\sum \mu_i = \mu}\left((\widetilde{h}_{a_1}^\vee)^{}_{\mu_1} \otimes \dots \otimes (\widetilde{h}_{a_n}^\vee)^{}_{\mu_n}\right)(\emptyset)\right)_{S_n}.
    \end{equation*}
    Now, $\widetilde{h}_a^\vee$ is the coequaliser of the two maps described in \eqref{eqn:021}. So $(\widetilde{h}_a^\vee)^{}_\mu$ is itself the coequaliser (in $\widehat{\mathcal{A}}$) of the two maps $(\pi \otimes h^\vee_a)^{}_\mu \to (h^\vee_a)^{}_\mu$ given by their degree $\mu$ parts. Since $\otimes$ preserves colimits, $(\widetilde{h}_{a_1}^\vee)^{}_{\mu_1} \otimes \dots \otimes (\widetilde{h}_{a_n}^\vee)^{}_{\mu_n}$ is the coequaliser of the two appropriate maps
    \begin{equation*}
        \bigoplus_{i}({h}_{a_1}^\vee)^{}_{\mu_1} \otimes \dots \otimes (\pi \otimes{h}_{a_i}^\vee)^{}_{\mu_i}\otimes\dots \otimes ({h}_{a_n}^\vee)^{}_{\mu_n} \longrightarrow {({h}_{a_1}^\vee)^{}_{\mu_1} \otimes \dots \otimes ({h}_{a_n}^\vee)^{}_{\mu_n}},
    \end{equation*}
    and hence $\left((\widetilde{h}_{a_1}^\vee)^{}_{\mu_1} \otimes \dots \otimes (\widetilde{h}_{a_n}^\vee)^{}_{\mu_n}\right)(\emptyset) \in \mathrm{mod}_R$ is a certain quotient of $\left(({h}_{a_1}^\vee)^{}_{\mu_1} \otimes \dots \otimes ({h}_{a_n}^\vee)^{}_{\mu_n}\right)(\emptyset)$. Given that $(h^\vee_a)^{}_\mu = \bigoplus_{\nu_1 + \nu_2 = \mu}a_{\nu_1\nu_2}$, we find that
    \begin{equation*}
        \left((\widetilde{h}_{a_1}^\vee)^{}_{\mu_1} \otimes \dots \otimes (\widetilde{h}_{a_n}^\vee)^{}_{\mu_n}\right)(\emptyset)
        \cong
        \bigoplus_{\nu_1^{(i)} + \nu_2^{(i)} = \mu_i} \mathcal{A}\left(\emptyset,\left((a_1)_{\nu^{(1)}_1\nu^{(1)}_2},\dots,(a_n)_{\nu^{(n)}_1\nu^{(n)}_2}\right)\right)\Big/ \sim
    \end{equation*}
    is the free $R$-module on the set of coloured string diagrams from $\emptyset$ to the tuple of arrows $(a_1,\dots,a_n)$ equipped with appropriate flavourings, modulo the following `local' flavour-changing relations:
    \begin{equation}\label{eqn:flavour-change}
        \vcenter{\hbox{\includegraphics[scale=1]{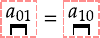}}} \quad \text{and} \quad
        \vcenter{\hbox{\includegraphics[scale=1]{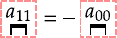}}}\;,
    \end{equation}
    where the arrow $a$ can be any one of $a_1,\dots,a_n$.

    Taking into account the $S_n$-action and putting everything together, we have
    \begin{equation*}
        \underline{\mathcal{C}}(1,S\widetilde{h}^\vee) \cong
        \left\langle
        \bigsqcup_n\bigsqcup_{a_1,\dots,a_n}\bigsqcup_{\mu_i,\nu_i}
        \mathrm{Diag}_{\overline{Q}}'\left(\emptyset,\left( (a_1)_{\mu_1\nu_1},\dots,(a_n)_{\mu_n\nu_n}\right) \right)
        \right\rangle_R \Big/ \sim,
    \end{equation*}
    the free $R$-module on the set of coloured string diagrams from $\emptyset$ to finite tuples of flavoured arrows, where a diagram $\emptyset \to \left( (a_1)_{\mu_1\nu_1},\dots,(a_n)_{\mu_n\nu_n}\right)$ has parity $\sum_i(\mu_i + \nu_i)$, modulo the relations \eqref{eqn:flavour-change} and
    \begin{equation}\label{eqn:023}
        \includegraphics[scale=1]{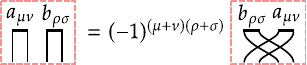}\;.
    \end{equation}
    Call flavoured arrows of the forms $a_{00}$ and $a_{11}$ \textit{even}, and of the forms $a_{01}$ and $a_{10}$ \textit{odd}. The total number of odd arrows in a diagram $\emptyset \to \left( (a_1)_{\mu_1\nu_1},\dots,(a_n)_{\mu_n\nu_n}\right)$ --- which we call its \textit{weight} --- is invariant under the relations, and $\underline{\mathcal{C}}(1,S\widetilde{h}^\vee)$ is additionally graded (as a superalgebra) by the nonnegative integer weight. The map $\mathrm{otr}$, which is necessarily zero on constant paths, sends
    \begin{equation*}
        \vcenter{\hbox{\includegraphics[scale=1]{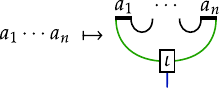}}}\;,
    \end{equation*}
    where the right hand side is merely diagrammatic \textit{notation}. The outermost strand (in green) is simply the coevalution map $1 \to c_i \otimes c_i^\vee$ for some vertex $i$, which has nonzero components
    \begin{equation*}
        \vcenter{\hbox{\includegraphics[scale=1]{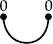}}}\quad 
        \text{and}\quad 
        \vcenter{\hbox{\includegraphics[scale=1]{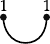}}}\;,
    \end{equation*}
    while the remaining strands are composites $\sigma \circ \mathrm{coev}$, which have nonzero components
    \begin{equation*}
        \vcenter{\hbox{\includegraphics[scale=1]{Diagrams/122.pdf}}}\quad 
        \text{and}\quad 
        \vcenter{\hbox{\includegraphics[scale=1]{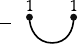}}}\;.
    \end{equation*}
    Keeping in mind the application of $\iota$, we get
    \begin{equation*}
        \vcenter{\hbox{\includegraphics[scale=1]{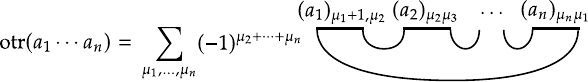}}}\;.
    \end{equation*}
    The lowest weight component of $\mathrm{otr}(a_1\cdots a_n)$ that is nonzero, $\mathrm{lw}(\mathrm{otr}(a_1\cdots a_n))$, has weight $1$ and is
    \begin{equation}\label{eqn:lowest-weight}
        \vcenter{\hbox{\includegraphics[scale=1]{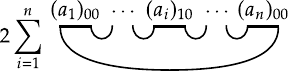}}}\;.
    \end{equation}

    Now, $S\Pi \mathcal{N} \cong \bigwedge \mathcal{N} \cong \langle \mathscr{B}\rangle_R$ is free as a module, with basis $\mathscr{B}$ a maximal collection of finite tuples of distinct cyclic paths whose underlying unordered tuples are distinct. There is then a natural decomposition $S\Pi\mathcal{N} \cong \langle \mathscr{B}'\rangle_R \oplus \langle \mathscr{B}'' \rangle_R$, where $\mathscr{B}' \subset \mathscr{B}$ consists of those tuples containing a constant path, and where $\mathscr{B}'' \subset \mathscr{B}$ consists of those tuples containing only nonconstant paths. The summand $\langle \mathscr{B}' \rangle_R$ is simply the ideal $\langle \{e_i\}_{i \in \overline{Q}_0} \rangle$, on which $\mathrm{otr}$ vanishes. Thus, it remains to show that $\mathrm{otr}$ is injective on $\langle\mathscr{B}'' \rangle_R$.
    
    By \eqref{eqn:lowest-weight}, the lowest weight of $\mathrm{otr}$ applied to a basis element $(a_1^1 \cdots a_{n_1}^1) \cdots (a_1^p \cdots a_{n_p}^p)$ is
    \begin{equation}\label{eqn:024}
        \vcenter{\hbox{\includegraphics[scale=1]{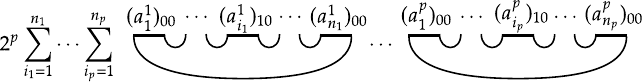}}}\;.
    \end{equation}
    By similar reasoning as in the proofs of Proposition \ref{prop:injectivity-of-tr} and Theorem \ref{thm:graded-necklace-BV-alg}, the map $\mathscr{B}'' \to \underline{\mathcal{C}}(1,S\widetilde{h}^\vee)$ sending $x \mapsto \mathrm{lw}(\mathrm{otr}(x))$ is then seen to be injective and have linearly independent image, where for this it helps to recall that we already reduced according to the relations \eqref{eqn:flavour-change} in expression \eqref{eqn:024} --- so that only the permutation action \eqref{eqn:023} needs to be considered. It follows that the map $\mathscr{B}'' \to \underline{\mathcal{C}}(1,S\widetilde{h}^\vee)$ inducing $\mathrm{otr}$ (not taking lowest weights) is itself injective with linearly independent image, and hence that $\mathrm{otr}$ is injective on $\langle \mathscr{B}''\rangle_R$.
\end{proof}

We finish this section by speculating on how one might arrange for $\mathrm{otr}$ to be injective. Recall that the construction on the representation side involves a triple $(\mathcal{C},c,\iota)$ consisting of a choice of $\otimes$-cocomplete symmetric monoidal $\Pi$-category and dualisable objects with odd involutions. Let us restrict ourselves to triples where the self-invertibility data of $\pi \in \mathcal{C}$ (the isomorphism $\pi \otimes \pi \isoto 1$) is simultaneously \textit{self-duality data}, for simplicity. For an injective $\mathrm{otr}$ map, we should try to choose the universal such triple, in which case there should not be any reason for the composites $\mathrm{otr}(e_i):\pi \to \pi \otimes c_i \otimes c_i^\vee \to c_i \otimes c_i^\vee \to 1$ to be zero. That is, we should take $\mathcal{C}$ to be the free $\otimes$-cocomplete $R$-linear symmetric monoidal category on
\begin{enumerate}
    \item A self-inverse object $\pi$, where the self-invertibility data is simultaneously self-duality data; and
    \item Dualisable objects $(c_i)_{i \in \overline{Q}_0}$ equipped with maps $(\iota_i:\pi \otimes c_i \to c_i)_{i \in \overline{Q}_0}$,
\end{enumerate}
subject to the conditions
\begin{enumerate}
    \item $\sigma_{\pi,\pi} = -\mathrm{id}$; and
    \item The composites $c_i \to \pi \otimes \pi \otimes c_i \to \pi \otimes c_i \to c_i$ are identities.
\end{enumerate}
Of course, we would need to argue that this category exists and, say, through a concrete description, show that the map $\mathrm{otr}$ is injective.

Let us sketch a slightly-more concrete description of what we expect this category to be. We will not prove it to be well-defined, nor show the expected injectivity of $\mathrm{otr}$ --- we leave this as an open question for further investigation. Nevertheless, consider finite tuples of elements of $(\overline{Q}_0 \times \{+,-\}) \sqcup \{\mathcolor{blue}{\text{blue}} = \mathcolor{blue}{\bullet}\}$, thought of as tuples of points where each point is either $\overline{Q}_0$-coloured and framed, or blue and unframed. Given two such tuples $x$ and $y$, and thinking of them as lying on the lower and upper boundaries of the unit square $[0,1]^2$, respectively, define an \textit{admissible diagram from $x$ to $y$} to consist of the following (Figure \ref{fig:admissible-diagram} gives an example and illustration):
\begin{enumerate}
    \item A (finite) set of oriented paths in $[0,1]^2$ that pairwise connect $\overline{Q}_0$-coloured points in $x \sqcup y$ of the same colour, and whose orientations respect their framings. These paths are naturally $\overline{Q}_0$-coloured.
    \item A finite set of oriented $\overline{Q}_0$-coloured loops in $[0,1]^2$.
    \item A finite set $X$ of marked points on these oriented $\overline{Q}_0$-coloured paths and loops. These points are coloured blue.
    \item A (finite) set of unoriented paths in $[0,1]^2$ which pairwise connect the blue points in $x \sqcup y\sqcup X$. These paths are declared blue.
    \item A finite set of unoriented blue loops in $[0,1]^2$.
\end{enumerate}
\begin{figure}
    \centering
    \includegraphics[scale=1]{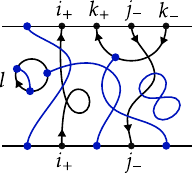}
    \caption{An admissible diagram from $(\mathcolor{blue}{\bullet},i_+,\mathcolor{blue}{\bullet},j_-,\mathcolor{blue}{\bullet})$ to $(\mathcolor{blue}{\bullet},i_+,k_+,j_-,k_-)$.}
    \label{fig:admissible-diagram}
\end{figure}

Call two admissible diagrams from $x$ to $y$ \textit{equivalent} if they are related by homotopy of oriented $\overline{Q}_0$-coloured paths and loops and \textit{regular} homotopy of unoriented blue paths and loops (where the marked points go along for the ride). We then let $\mathrm{AdDiag}_{\overline{Q}}$ be the symmetric monoidal category whose objects are these tuples and whose maps $x \to y$ are equivalence classes of admissible diagrams from $x$ to $y$, where the composition law and monoidal structure are given by vertical stacking and horizontal juxtaposition. Let $\mathcal{B}$ be the $R$-linear symmetric monoidal category obtained from $\mathrm{AdDiag}_{\overline{Q}}$ by taking free $R$-modules of its hom sets and modding out by the following local relations:
\begin{equation*}
    \vcenter{\hbox{\includegraphics[scale=1]{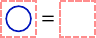}}}\;, \quad
    \vcenter{\hbox{\includegraphics[scale=1]{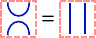}}}\;, \quad
    \vcenter{\hbox{\includegraphics[scale=1]{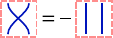}}}\;, \quad \text{and} \quad 
    \vcenter{\hbox{\includegraphics[scale=1]{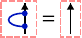}}}\;.
\end{equation*}
The first three relations are to give $\mathcal{B}$ a parity structure with $\pi = (\mathcolor{blue}{\bullet})$ (which is already self-dual), while the last then ensures that the map $\pi \otimes (i_+) = (\mathcolor{blue}{\bullet},i_+) \to (i_+)$ given by
\begin{equation}\label{eqn:025}
    \includegraphics[scale=1]{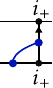}\;
\end{equation}
is an odd involution for each vertex $i$. Under the (enriched) Yoneda embedding $\mathcal{B} \hookrightarrow \widehat{\mathcal{B}}$, taking $c_i$ to be $(i_+)$ --- which is dualisable with dual $(i_-)$ --- and $\iota_i$ to be the odd involution \eqref{eqn:025}, for each $i$, we expect the triple $(\mathcal{C} = \widehat{\mathcal{B}},c,\iota)$ to be the free one so-described and yield injective $\mathrm{otr}:S\Pi\mathcal{N} \to \underline{\mathcal{C}}(1,S\widetilde{h}^\vee)$.

\newpage
\appendix
\section{}\label{Appendix A}
\begin{proof}[Proof of Proposition \ref{prop:explicit-diffop}]
    For $I \subset \{1,\dots,N\}$, denote by $X_I$ the map $a^{\otimes N} \to a$ given by
    \begin{equation*}
        \vcenter{\hbox{\includegraphics[scale=1]{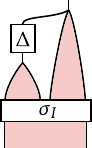}}}\;.
    \end{equation*}
    It follows from Proposition \ref{prop:diffop-condition} that for any $J \subset \{1,\dots,N\}$ with $|J| \geq k+1$ we have $\sum_{I \subset J}(-1)^{|I|}X_I = 0$. Thus,
    \begin{equation}\label{eqn:001}
        \sum_{\substack{J\subset \{1,\dots,N\}\\ |J| \geq k+1}} (-1)^{|J|}\sum_{I \subset J}(-1)^{|I|} X_I = 0.
    \end{equation}
    Reversing the order of summation, we can write the left hand side of \eqref{eqn:001} as
    \begin{equation*}
        \sum_{I  \subset\{1,\dots,N\}} c_I (-1)^{|I|}X_I,\quad\text{where}\quad
        c_I = \sum_{\substack{J \subset \{1,\dots,N\}\\|J| \geq k+1,\; J \supset I}}(-1)^{|J|}.
    \end{equation*}
    Notice that then
    \begin{equation*}
        c_I = \sum_{p=k+1}^N(-1)^p |\{J \subset \{1,\dots,N\} \; | \; |J| = p, \; J \supset I\}| =  \sum_{p=k+1}^N(-1)^p \binom{N-|I|}{N-p} = \sum_{q=0}^{N-k-1}(-1)^{N+q} \binom{N-|I|}{q},
    \end{equation*}
    where in the last equality we reindexed with $q = N-p$. We now see that there are three cases:
    \begin{enumerate}
        \item[(1)] `$|I| = N$'. We immediately see that $c_I = (-1)^N$.
        \item[(2)] `$k+1 \leq |I| < N$'. Since $N-|I| \leq N-k-1$ and $N-|I| > 0$, we have
            \begin{equation*}
                c_I = (-1)^N \sum_{q=0}^{N-|I|}(-1)^q \binom{N-|I|}{q} = 0.
            \end{equation*}
        \item[(3)] `$|I| \leq k$'. Then $N-|I| > N-k-1 \geq 0$, so
            \begin{equation*}
                c_I = (-1)^N\sum_{q=0}^{N-k-1}(-1)^{q} \binom{N-|I|}{q} = (-1)^{k+1} \binom{N-|I|-1}{N-k-1},
            \end{equation*}
            by a well-known binomial coefficients identity.
    \end{enumerate}
    Thus, the left hand side of \eqref{eqn:001} (which vanishes) is
    \begin{equation*}
        X_{\{1,\dots,N\}} + \sum_{\substack{I \subset \{1,\dots,N\}\\ |I| \leq k}}(-1)^{k+1+|I|}\binom{N-|I|-1}{N-k-1}X_I,
    \end{equation*}
    yielding the result.
\end{proof}

\begin{lem}\label{lem:for-!existence-of-diffops}
    Let $k \geq 0$, and write $\mathscr{A} = \{1,\dots,k+1\}$. For all collections of finite sets $(\mathscr{B}_i)_{i \in \mathscr{A}}$ and subsets $J \subset \mathscr{B}$ with order $|J| \leq k$,
    \begin{equation}\label{eqn:004}
        \sum_{\substack{I \subset \mathscr{A}\\ \mathscr{B}_I \supset J}}(-1)^{|I|}\binom{|\mathscr{B}_I| - |J| - 1}{k - |J|} = 0,
    \end{equation}
    where $\mathscr{B}_I := \bigsqcup_{i \in I} \mathscr{B}_i$, $\mathscr{B} := \mathscr{B}_{\mathscr{A}}$, and $\binom{-1}{\beta}$ is defined to be $(-1)^\beta$ for all $\beta \geq 0$. 
\end{lem}
\begin{proof}
    We first prove the simpler identity
    \begin{equation}\label{eqn:005}
        \sum_{\substack{I \subset \mathscr{A}\\ \mathscr{B}_I \supset J}}(-1)^{|I|}\binom{|\mathscr{B}_I| - |J|}{k - |J|} = 0.
    \end{equation}
    Let $X =\{R \subset \mathscr{B} \; | \; R \supset J, |R| = k\}$, and for each $i \in \mathscr{A}$ let $X_i = \{R \in X \; | \; R \cap \mathscr{B}_i = \emptyset\}$. Note that $X = \bigcup_{i \in \mathscr{A}} X_i$, since if $R \cap \mathscr{B}_i \neq \emptyset$ for all $i$, then $|R| \geq |\mathscr{A}| = k+1$. By inclusion-exclusion we have $\sum_{I \subset \mathscr{A}}(-1)^{|I|}\left|\bigcap_{i \in I} X_i\right| = 0.$ Now,
    \begin{equation*}
    \textstyle
        \left| \bigcap_{i \in I} X_i\right| = |\{R \subset \mathscr{B} \setminus \mathscr{B}_I \; | \; R \supset J,|R|=k\}|
        =
        \begin{cases}
            0, \; \text{if $\mathscr{B}_I \cap J \neq \emptyset$}\\
            \binom{|\mathscr{B}\setminus \mathscr{B}_I| - |J|}{k - |J|}, \;\text{if $\mathscr{B}_I \cap J = \emptyset$},
        \end{cases}
    \end{equation*}
    so in fact the inclusion-exclusion result gives
    \begin{equation*}
        0 = \sum_{\substack{I \subset \mathscr{A}\\ \mathscr{B}_I \cap J = \emptyset}}(-1)^{|I|}\binom{|\mathscr{B}\setminus \mathscr{B}_I| - |J|}{k - |J|}
        = \sum_{\substack{I \subset \mathscr{A}\\ \mathscr{B}_I \supset J}}(-1)^{|I| + k + 1}\binom{|\mathscr{B}_I| - |J|}{k - |J|},
    \end{equation*}
    where we re-indexed $I \leftrightarrow \mathscr{A} \setminus I$, showing that \eqref{eqn:005} holds.
    
    We now turn our attention to the original claim, which we will prove by induction on $k$. For the base case $k=0$ we necessarily have $J = \emptyset$, so that \eqref{eqn:004} is
    \begin{equation*}
        \binom{-1}{0} - \binom{|\mathscr{B}|-1}{0} = 0,
    \end{equation*}
    which is true.
    Consider now $k \geq 1$. Notice that Pascal's identity $\binom{\alpha - 1}{\beta} = \binom{\alpha}{\beta} - \binom{\alpha - 1}{\beta - 1}$ holds for all $\alpha,\beta \geq 0$, given that $\binom{-1}{\beta} := (-1)^\beta$ for $\beta \geq 0$ (as before), and $\binom{\alpha}{-1} := 0$ for $\alpha \geq -1$. Using this together with \eqref{eqn:005}, we get
    \begin{equation}\label{eqn:006}
        \sum_{\substack{I \subset \mathscr{A}\\ \mathscr{B}_I \supset J}}(-1)^{|I|}\binom{|\mathscr{B}_I| - |J| - 1}{k - |J|} 
        =
        -\sum_{\substack{I \subset \mathscr{A}\\ \mathscr{B}_I \supset J}}(-1)^{|I|}\binom{|\mathscr{B}_I| - |J| - 1}{k - |J| -1}.
    \end{equation}
    Consider now three cases:
    \begin{enumerate}
        \item[(1)] `$|J| = k$'. It is immediate that \eqref{eqn:006} vanishes.
        \item[(2)] `$|J| \leq k-1$ and $J = \mathscr{B}$'. Then $\mathscr{B}_i = \emptyset$ for some $i$. Without loss of generality, suppose that $$\mathscr{B}_i \neq \emptyset \iff i \in \{1,\dots,n\} \subsetneq \mathscr{A},$$ for some $n \leq k$. We have $\mathscr{B}_I \supset J \iff I \supset \{1,\dots,n\}$, and so \eqref{eqn:006} becomes
        \begin{equation*}
            -\sum_{\substack{I \subset \mathscr{A}\\ I \supset \{1,\dots,n\}}}(-1)^{|I|}\binom{-1}{k - |J| -1}
            =
            \sum_{\substack{I \subset \mathscr{A}\\ I \supset \{1,\dots,n\}}}(-1)^{|I| + k + |J|}
            =
            (-1)^{|J|+1} \sum_{I \subset \{n+1,\dots,k+1\}} (-1)^{|I|}
            = 0.
        \end{equation*}
        \item[(3)] `$|J| \leq k-1$ and $J \subsetneq \mathscr{B}$'. Let $\star \in \mathscr{B}\setminus J$. Then \eqref{eqn:006} can be written as
        \begin{equation*}
            -\underbrace{
            \sum_{\substack{I \subset \mathscr{A}\\ \mathscr{B}_I \supset J\sqcup \{\star\}}}(-1)^{|I|}\binom{|\mathscr{B}_I| - |J\sqcup\{\star\}|}{k - |J\sqcup\{\star\}|}
            }_{(\ast)}
            -
            \underbrace{
            \sum_{\substack{I \subset \mathscr{A}\\ \mathscr{B}_I \supset J, \,\mathscr{B}_I \notni \star}}(-1)^{|I|}\binom{|\mathscr{B}_I| - |J| - 1}{k - |J| -1}
            }_{(\dagger)}.
        \end{equation*}
        As $|J \sqcup \{\star\}| \leq k$, applying \eqref{eqn:005} (with $J\sqcup \{\star\}$ instead of $J$) shows that $(\ast)$ vanishes. As for $(\dagger)$, we consider two further cases. Suppose without loss of generality that $\star \in \mathscr{B}_{k+1}$.
        \begin{enumerate}
            \item[(i)]`$\mathscr{B}_{k+1} \cap J \neq \emptyset$'. We have $\mathscr{B}_I \supset J \implies \mathscr{B}_I \ni \star$, so $(\dagger)$ vanishes.
            \item[(ii)]`$\mathscr{B}_{k+1} \cap J = \emptyset$'. Note that $J \subset \mathscr{B}_{\{1,\dots,k\}}$. Since $\mathscr{B}_I \notni \star \iff I \subset \{1,\dots,k\}$, we have
            \begin{equation*}
                (\dagger) = \sum_{\substack{I \subset \{1,\dots,k\}\\ \mathscr{B}_I \supset J}}(-1)^{|I|}\binom{|\mathscr{B}_I| - |J| - 1}{(k-1) - |J|},
            \end{equation*}
            which shows that the proof is completed by induction.
        \end{enumerate}
    \end{enumerate}
\end{proof}
\bibliographystyle{alpha}
\bibliography{Bibliography}
\end{document}